\documentclass[11pt,a4paper]{article}
\usepackage[margin=28mm]{geometry}
\usepackage{amsmath,amssymb,amsthm,mathtools,mathrsfs}
\usepackage{tikz-cd}
\usepackage{bm}
\usepackage{enumitem}
\usepackage[hidelinks]{hyperref}
\newtheorem{thm}{Theorem}[section]
\newtheorem{prop}[thm]{Proposition}
\newtheorem{cor}[thm]{Corollary}
\newtheorem{lem}[thm]{Lemma}
\theoremstyle{definition}
\newtheorem{dfn}[thm]{Definition}
\newtheorem{ex}[thm]{Example}
\theoremstyle{remark}
\newtheorem{rem}[thm]{Remark}
\newcommand{\Z}{\mathbb Z}
\newcommand{\F}{\mathbb F}
\newcommand{\D}{\mathcal D}
\newcommand{\B}[1]{\Omega(#1)}
\newcommand{\NS}{\operatorname{NS}(G)}
\setlist{itemsep=2pt,topsep=4pt}
\allowdisplaybreaks
\title{Normal-Core Compression for Units in Burnside Rings}
\author{Masahiro Wakatake}
\date{}

\begin{document}
\maketitle

\begin{abstract}
Let $G$ be a finite group. We study the additive map on the Burnside ring of $G$ that sends an orbit $[G/H]$ to $[G/\operatorname{core}_G(H)]$, which we call the \emph{normal-core compression}, and investigate its behavior on units. For a normal subgroup $N\trianglelefteq G$, we show that the cumulative $N$-core coefficient is equal to the sum of the orbit-basis coefficients of the $N$-fixed-point element over $G/N$. Consequently, for a unit this coefficient takes only the values $-1,0,1$, while the individual core coefficients are recovered by M\"obius inversion on the lattice of normal subgroups. Using Yoshida's criterion, we further express the cumulative core coefficient in terms of the $N$-mark and a linear character of $G/N$. This yields a necessary and sufficient condition for the existence of a unit with nonzero $N$-core coefficient in the normal partial Burnside ring: $G/N$ must be an elementary abelian $2$-group. On the other hand, the units mapped to $1_{\B G}$ by the core compression, called \emph{core-trivial units}, are characterized by the condition that all marks at core-cyclic subgroups are equal to $1$. Using primitive idempotents of the rational Burnside ring and Yoshida's criterion, we realize the group of core-trivial units as the kernel of an $\F_2$-linear defect map. Finally, we obtain a splitting along normal quotients and a direct-sum decomposition along direct products for this defect map. Combining the direct-product decomposition with an explicit nontrivial example for $S_4$, we show that the ranks of core-trivial unit groups are unbounded among finite groups.
\end{abstract}

\begin{center}
\end{center}

\section{Introduction and main results}

Let $G$ be a finite group. The unit group of the Burnside ring $\B G$ is a finite elementary abelian $2$-group, and every mark of a unit is equal to $\pm1$. Classical work on unit groups of Burnside rings includes Matsuda~\cite{Matsuda1982}. Yoshida characterized those $\{\pm1\}$-valued mark vectors that arise from Burnside units by a homomorphism condition on the Weyl group of each subgroup~\cite{Yoshida1990}. We refer to this homomorphism condition as \emph{Yoshida's criterion}. For the partial Burnside ring relative to the family of normal subgroups, explicit generators corresponding to normal subgroups of index $2$ are also known~\cite[Proposition~4.6]{Wakatake2018}.

Throughout the paper, we write $\operatorname{NS}(G)$ for the set of normal subgroups of $G$ and focus on the normal core
\[
 \operatorname{core}_G(H)
 =\bigcap_{g\in G}gHg^{-1}
\]
of a subgroup $H\leq G$. On the orbit basis, define
\[
 \mathfrak n_G([G/H])
 :=[G/\operatorname{core}_G(H)],
\]
and extend this assignment additively to obtain
\[
 \mathfrak n_G:\B G\longrightarrow\B{G,\operatorname{NS}(G)}.
\]
We call this additive map the \emph{normal-core compression}, or simply the \emph{core compression}. It aggregates stabilizer information according to normal cores. The two main questions of this paper are: what information about Burnside units is retained by the core compression, and how can the units that become trivial under core compression be described?

For the first question, let $c_N(x)$ denote the $N$-core coefficient of $x\in\B G$, and let $s_N(x)$ denote its cumulative sum. If $N\trianglelefteq G$ and $\operatorname{coeffsum}$ denotes the sum of orbit-basis coefficients, then we prove
\[
 s_N(x)=\operatorname{coeffsum}_{G/N}(x^N)
\]
(Theorem~\ref{thm:core-cumulative}). Hence, for $u\in\B G^\times$,
\[
 s_N(u)\in\{-1,0,1\},
\]
and the individual core coefficients are recovered by M\"obius inversion on the lattice of normal subgroups:
\[
 c_N(u)=\sum_{N\leq M\trianglelefteq G}\mu_{\operatorname{NS}(G)}(N,M)s_M(u)
\]
(Proposition~\ref{prop:mobius}). Thus, even when individual core coefficients can be large, their information is encoded by a family of three-valued cumulative coefficients.

Yoshida's criterion sharpens this three-valued description. For $N\trianglelefteq G$ and $u\in\B G^\times$, define
\[
 \chi_N^u(gN)
 :=
 \frac{\varphi_{\langle N,g\rangle}^G(u)}{\varphi_N^G(u)}.
\]
Then $\chi_N^u:G/N\to\{\pm1\}$ is a group homomorphism, and
\[
 s_N(u)=
 \begin{cases}
  \varphi_N^G(u),&\chi_N^u=1_{G/N},\\
  0,&\chi_N^u\neq1_{G/N}
 \end{cases}
\]
(Theorem~\ref{thm:yoshida-dichotomy}). Hence the cumulative core coefficient is determined by the $N$-mark together with a linear character of $G/N$. Combined with M\"obius inversion, this also yields information on individual core coefficients.

As an application, we study the normal partial Burnside ring. For a normal subgroup $N\trianglelefteq G$, there exists a unit $u\in\B{G,\NS}^{\times}$ with
\[
 c_N(u)\neq0
\]
if and only if
\[
 G/N\text{ is an elementary abelian }2\text{-group}
\]
(Theorem~\ref{thm:normal-pbr-nonzero-core}). Moreover, inflation and fixed points reduce the core-coefficient problem to normal quotients, and the core compression is also compatible with direct products (Propositions~\ref{prop:inflation-core} and~\ref{prop:external-product-core}).

For the second question, namely the units that become trivial under core compression, define
\[
 \mathcal U_{\mathrm{core}}(G)
 :=\{u\in\B G^\times\mid\mathfrak n_G(u)=1_{\B G}\}.
\]
We call its elements \emph{core-trivial units}. A subgroup $H\leq G$ is called \emph{core-cyclic} if $H/\operatorname{core}_G(H)$ is cyclic. We prove
\[
 \mathcal U_{\mathrm{core}}(G)
 =
 \left\{
 u\in\B G^\times\ \middle|\
 \begin{array}{c}
 \varphi_H^G(u)=1\text{ for every}\\
 \text{core-cyclic subgroup }H\leq G
 \end{array}
 \right\}
\]
(Theorem~\ref{thm:core-trivial-subgroup}).

To describe this group by finite linear algebra, let $e_H^G$ denote the primitive idempotent of $\mathbb Q\B G$ corresponding to $(H)$. Write $\operatorname{Sub}_{\mathrm{ncc}}(G)$ for the set of non-core-cyclic subgroups. For a set $S$ of conjugacy classes of such subgroups, put
\[
 e_S:=\sum_{(H)\in S}e_H^G
\]
and
\[
 \mathcal S_2(G)
 :=\left\{S\subseteq\operatorname{Sub}_{\mathrm{ncc}}(G)^c\ \middle|\ 2e_S\in\B G\right\}.
\]
Then
\[
 \mathcal U_{\mathrm{core}}(G)
 =\{1_{\B G}-2e_S\mid S\in\mathcal S_2(G)\},
\]
and $\mathcal S_2(G)$ is an $\F_2$-vector space under symmetric difference (Theorem~\ref{thm:core-trivial-idempotents}).

Next define
\[
 V_G:=\F_2^{\operatorname{Sub}_{\mathrm{ncc}}(G)^c}.
\]
Using the homomorphism condition in Yoshida's criterion, we construct a defect map
\[
 \operatorname{Def}_G:V_G\longrightarrow\mathfrak D_G,
\]
where $\mathfrak D_G$ is a direct sum of finite-dimensional local defect spaces attached to the Weyl groups $W_G(K)=N_G(K)/K$. Our main linear criterion is
\[
 \ker\operatorname{Def}_G=\mathcal S_2(G)
\]
(Theorem~\ref{thm:yoshida-defect-kernel}). Consequently, determining the core-trivial units reduces to computing the kernel of a finite-dimensional $\F_2$-linear map arising from Yoshida's criterion.

Finally, we establish decompositions of the defect map with respect to normal quotients and direct products. If $N\trianglelefteq G$, linear maps corresponding to inflation and fixed points give
\[
 \mathcal S_2(G)
 =\operatorname{Inf}\bigl(\mathcal S_2(G/N)\bigr)
 \oplus
 \bigl(\mathcal S_2(G)\cap\ker\operatorname{Fix}\bigr)
\]
(Theorem~\ref{thm:yoshida-defect-normal-quotient-decomposition}). For direct products, we likewise obtain a direct-sum decomposition into the components coming from the individual factors and a residual component annihilated by all fixed-point maps (Theorem~\ref{thm:yoshida-defect-direct-product-decomposition}). These decompositions transfer directly to $\mathcal U_{\mathrm{core}}(G)$. In addition, $S_4$ has a nontrivial core-trivial unit (Example~\ref{ex:s4-core-trivial-unit}), and hence direct powers $S_4^m$ satisfy
\[
 \operatorname{rank}\bigl(\mathcal U_{\mathrm{core}}(S_4^m)\bigr)\geq m.
\]
Thus the ranks of core-trivial unit groups are unbounded among finite groups.

Section~2 fixes the basic notation and recalls Burnside rings, collections, partial Burnside rings, and marks. Section~3 introduces core coefficients and cumulative core coefficients and studies their fixed-point description, M\"obius inversion, compatibility with normal quotients, and examples. Section~4 refines the cumulative core coefficients by means of Yoshida characters. Section~5 gives applications to normal partial Burnside rings and a basic characterization of core-trivial units. Section~6 uses primitive idempotents of the rational Burnside ring and Yoshida's criterion to linearize the core-trivial units, realize them as the kernel of a defect map, and establish decompositions along normal quotients and direct products.

\section{Preliminaries on Burnside rings and partial Burnside rings}

In this section we recall the basic facts about Burnside rings, collections, partial Burnside rings, and mark homomorphisms, and fix the notation used throughout the paper.

\subsection{Basic notation}

Throughout, $G$ denotes a finite group. We write $H\trianglelefteq G$ if a subgroup $H\leq G$ is normal in $G$, and $(H)$ for the $G$-conjugacy class of $H$. We also write $\operatorname{Sub}(G)$ for the set of all subgroups of $G$.

For a finite set $X$, we use $\#X$ for its cardinality. For a finite group $G$ and a subgroup $H\leq G$, we write $|G|$ for the order of $G$ and $|G:H|$ for the index of $H$ in $G$.

If $\D$ is a family of subgroups closed under $G$-conjugacy, write
\[
\D^c:=\{(H)\mid H\in\D\}
\]
for the set of $G$-conjugacy classes represented in $\D$.

For a ring $R$, denote its multiplicative identity by $1_R$ and its unit group by $R^\times$. We write $C_2$ for the cyclic group of order $2$. For a group $Q$, let
\[
1_Q:Q\longrightarrow C_2,
\qquad
q\longmapsto 1
\]
denote the trivial homomorphism. A finite elementary abelian $2$-group $A$ is naturally an $\F_2$-vector space. We define
\[
\operatorname{rank}(A):=\dim_{\F_2}A.
\]
Thus
\[
A\cong C_2^{\,\operatorname{rank}(A)}.
\]

\subsection{Burnside rings}

For a finite group $G$, the isomorphism classes of finite $G$-sets form a commutative semiring under disjoint union and Cartesian product with the diagonal $G$-action. Its Grothendieck ring is the \emph{Burnside ring} of $G$, denoted by $\B G$. The element represented by a finite $G$-set $X$ is written $[X]$, and
\[
1_{\B G}=[G/G].
\]

Every finite $G$-set decomposes uniquely as a disjoint union of orbits, and every transitive finite $G$-set is isomorphic to $G/H$ for some subgroup $H\leq G$. Moreover, $G/H$ and $G/K$ are isomorphic as $G$-sets if and only if $H$ and $K$ are conjugate in $G$. Hence
\[
\B G
=
\bigoplus_{(H)\in\operatorname{Sub}(G)^c}\Z[G/H].
\]
The product of two orbit-basis elements is given by the double-coset formula
\[
[G/H][G/K]
=
\sum_{HgK\in H\backslash G/K}[G/(H\cap{}^gK)],
\qquad
{}^gK=gKg^{-1}.
\]

\subsection{Collections and partial Burnside rings}

\begin{dfn}[Collection]\label{dfn:collection}
A family of subgroups $\D\subseteq\operatorname{Sub}(G)$ is called a \emph{collection} of $G$ if it satisfies the following three conditions:
\begin{enumerate}[label=\textup{(\roman*)}]
 \item $G\in\D$;
 \item if $H\in\D$ and $g\in G$, then ${}^gH\in\D$;
 \item if $H,K\in\D$, then $H\cap K\in\D$.
\end{enumerate}
\end{dfn}

The additive subgroup spanned by the orbit-basis elements with stabilizers in $\D$,
\[
\B{G,\D}
=
\bigoplus_{(H)\in\D^c}\Z[G/H],
\]
is called the \emph{partial Burnside ring relative to $\D$}. By the double-coset formula and closure of $\D$ under conjugation and intersection, this is a subring of $\B G$ with the same identity element,
\[
1_{\B G}=[G/G].
\]
In particular,
\[
\B{G,\operatorname{Sub}(G)}=\B G.
\]

\subsection{Mark homomorphisms}

For a subgroup $K\leq G$ and a finite $G$-set $X$, let
\[
X^K:=\{x\in X\mid kx=x\text{ for all }k\in K\}
\]
denote the set of $K$-fixed points.

For a collection $\D$ and an arbitrary subgroup $K\leq G$, fixed-point cardinality defines a ring homomorphism
\[
\varphi_K^{\D}:\B{G,\D}\longrightarrow\Z,
\qquad
[X]\longmapsto\#X^K.
\]
Here we do not assume that $K\in\D$.

Taking the marks corresponding to the $G$-conjugacy classes represented in $\D$ gives a ring homomorphism
\[
\varphi^{\D}:\B{G,\D}\longrightarrow
\prod_{(K)\in\D^c}\Z,
\qquad
x\longmapsto
\bigl(\varphi_K^{\D}(x)\bigr)_{(K)\in\D^c}.
\]
We call this the \emph{mark homomorphism relative to $\D$}. It is injective~\cite[\S 2.2]{Wakatake2018}. Therefore we may identify
\[
\B{G,\D}
\cong
\operatorname{Im}\bigl(\varphi^{\D}\bigr)
\subseteq
\prod_{(K)\in\D^c}\Z.
\]

Moreover, $x\in\B{G,\D}$ is a unit if and only if
\[
\varphi_K^{\D}(x)\in\{-1,1\}
\qquad
(K\in\D).
\]
In that case, $x^2=1_{\B G}$.

When $\D=\operatorname{Sub}(G)$, we abbreviate
\[
\varphi_K^G:=\varphi_K^{\operatorname{Sub}(G)},
\qquad
\varphi^G:=\varphi^{\operatorname{Sub}(G)}.
\]

\subsection{The family of normal subgroups}

The family
\[
\NS
=
\{N\leq G\mid N\trianglelefteq G\}
\]
is a collection. We call the corresponding partial Burnside ring $\B{G,\NS}$ the \emph{normal partial Burnside ring}.

For normal subgroups $H,K\trianglelefteq G$,
\[
[G/H][G/K]
=
|G:HK|[G/(H\cap K)].
\]
For a normal subgroup $H\trianglelefteq G$ and an arbitrary subgroup $L\leq G$,
\[
\varphi_L^{\NS}([G/H])
=
\begin{cases}
|G:H|,&L\leq H,\\
0,&L\nleq H.
\end{cases}
\]

The following generating theorem for the unit group of the normal partial Burnside ring is known~\cite[Proposition~4.6]{Wakatake2018}:
\[
\B{G,\NS}^{\times}
=
\left\langle
-1_{\B G},\;
1_{\B G}-[G/N]
\ \middle|\
N\trianglelefteq G,\ |G:N|=2
\right\rangle.
\]
For each normal subgroup $N$ of index $2$, put
\[
u_N:=1_{\B G}-[G/N].
\]
Then $u_N^2=1_{\B G}$, and $-1_{\B G}$ together with all the $u_N$ are independent. Hence
\[
\B{G,\NS}^{\times}
\cong
C_2^{\,1+\#\{N\trianglelefteq G\mid |G:N|=2\}},
\]
and, in particular,
\[
\operatorname{rank}\bigl(\B{G,\NS}^{\times}\bigr)
=
1+\#\{N\trianglelefteq G\mid |G:N|=2\}.
\]
We shall use this generating theorem below.

\section{Normal-core coefficients and their cumulative sums}
\subsection{The Burnside average and cumulative core coefficients}

\begin{dfn}[Coefficient sum]
If
\[
x=\sum_{(H)\in\D^c}a_H[G/H]\in\B{G,\D},
\]
define
\[
\operatorname{coeffsum}_G(x)
:=
\sum_{(H)\in\D^c}a_H.
\]
\end{dfn}

Thus $\operatorname{coeffsum}_G(x)$ is the sum of the coefficients with respect to the orbit basis. It differs from the augmentation $[G/H]\mapsto |G:H|$, which gives the cardinality of a finite $G$-set, and in general it is not a ring homomorphism.

We first express the coefficient sum as a Burnside average of marks at cyclic subgroups and then interpret the cumulative coefficients arising from normal-core compression over normal quotients.

Since a collection $\D$ is closed under conjugation and finite intersections, if $H\in\D$, then
\[
\operatorname{core}_G(H)\in\D\cap\NS.
\]

\begin{dfn}[Core coefficients and cumulative core coefficients]
Let
\[
x=\sum_{(H)\in\D^c}a_H[G/H]\in\B{G,\D}.
\]
For each $N\in\D\cap\NS$, define
\[
c_N(x)
:=
\sum_{\substack{(H)\in\D^c\\ \operatorname{core}_G(H)=N}}a_H,
\qquad
s_N(x)
:=
\sum_{\substack{M\in\D\cap\NS\\ N\leq M}}c_M(x).
\]
We call these the \emph{$N$-core coefficient} and the \emph{cumulative $N$-core coefficient} of $x$, respectively.

We call the additive homomorphism
\[
\mathfrak{n}_{G,\D}:
\B{G,\D}
\longrightarrow
\B{G,\D\cap\NS},
\qquad
x\longmapsto
\sum_{N\in\D\cap\NS}c_N(x)[G/N]
\]
the \emph{normal-core compression} with respect to $G$ and $\D$. We abbreviate
\[
\mathfrak n_G
:=
\mathfrak n_{G,\operatorname{Sub}(G)}.
\]
\end{dfn}

If $N\trianglelefteq G$, then for every subgroup $H\leq G$,
\begin{equation}\label{eq:normal-core-containment}
N\leq \operatorname{core}_G(H)
\quad\Longleftrightarrow\quad
N\leq H.
\end{equation}
Therefore
\begin{equation}\label{eq:sN-basis}
s_N(x)
=
\sum_{\substack{(H)\in\D^c\\ N\leq H}}a_H.
\end{equation}

For $N\in\D\cap\NS$, put
\[
\D_{\geq N}
:=
\{H\in\D\mid N\leq H\}
\]
and
\[
\D_{\geq N}/N
:=
\{H/N\mid H\in\D_{\geq N}\}.
\]
Then $\D_{\geq N}/N$ is a collection of the quotient group $G/N$.

Since $N\trianglelefteq G$, the fixed-point set $X^N$ is $G$-stable, and $N$ acts trivially on $X^N$. Hence $X^N$ is naturally a $G/N$-set. Taking $N$-fixed points therefore defines a unital ring homomorphism
\[
(-)^N:
\B{G,\D}
\longrightarrow
\B{G/N,\D_{\geq N}/N},
\qquad
[X]\longmapsto[X^N].
\]
On the orbit basis,
\begin{equation}\label{eq:normal-fixed-orbit}
[G/H]^N
=
\begin{cases}
[(G/N)/(H/N)],&N\leq H,\\
0,&N\nleq H.
\end{cases}
\end{equation}

\begin{prop}[Burnside average]
For every $x\in\B{G,\D}$,
\[
\operatorname{coeffsum}_G(x)
=
\frac{1}{|G|}
\sum_{g\in G}
\varphi_{\langle g\rangle}^{\D}(x).
\]
In particular, if $u\in\B{G,\D}^{\times}$, then
\[
\operatorname{coeffsum}_G(u)\in\{-1,0,1\}.
\]
\end{prop}

\begin{proof}
For a transitive $G$-set $G/H$, Burnside's lemma identifies the right-hand side with the number of $G$-orbits of $G/H$, namely $1$. The first identity follows by $\Z$-linearity.

If $u$ is a unit, then
\[
\varphi_{\langle g\rangle}^{\D}(u)\in\{-1,1\}
\]
for every $g\in G$. Hence $\operatorname{coeffsum}_G(u)$ is an integer in the interval $[-1,1]$.
\end{proof}

\begin{thm}[Fixed-point description of cumulative core coefficients]
\label{thm:core-cumulative}
For every $x\in\B{G,\D}$ and $N\in\D\cap\NS$,
\[
s_N(x)
=
\operatorname{coeffsum}_{G/N}(x^N).
\]
In particular, if $u\in\B{G,\D}^{\times}$, then
\[
s_N(u)
=
\frac{1}{|G/N|}
\sum_{gN\in G/N}
\varphi_{\langle gN\rangle}^{\D_{\geq N}/N}(u^N)
\in\{-1,0,1\}.
\]
\end{thm}

\begin{proof}
Comparing~\eqref{eq:sN-basis} with the expansion of $x^N$ in the orbit basis of $G/N$ gives
\[
s_N(x)
=
\operatorname{coeffsum}_{G/N}(x^N).
\]

The fixed-point map is a unital ring homomorphism, so $u^N$ is a unit of $\B{G/N,\D_{\geq N}/N}$. Applying the Burnside average to $G/N$ and $u^N$ yields the second formula.
\end{proof}

\begin{cor}[Extreme values of cumulative core coefficients]
\label{cor:zero-criterion}
Let $u\in\B{G,\D}^{\times}$ and $N\in\D\cap\NS$. Then $s_N(u)=0$ if and only if $|G:N|$ is even and the numbers of $+1$ and $-1$ terms occurring in the average of Theorem~\ref{thm:core-cumulative} are equal. Moreover, $s_N(u)=1$ if and only if every term in the average is $1$, and similarly for $s_N(u)=-1$.

In particular, if $|G:N|$ is odd, then
\[
s_N(u)
=
\varphi_N^{\D}(u)
\in\{-1,1\},
\]
and, for every $g\in G$,
\[
\varphi_{\langle gN\rangle}^{\D_{\geq N}/N}(u^N)
=
\varphi_N^{\D}(u).
\]
\end{cor}

\begin{proof}
This follows because the right-hand side in Theorem~\ref{thm:core-cumulative} is the average of $|G:N|$ numbers, each equal to $\pm1$.
\end{proof}

\subsection{M\"obius inversion on the family of normal subgroups}

Let $\mu$ denote the M\"obius function of the finite poset $\D\cap\NS$.

\begin{prop}[M\"obius inversion]
\label{prop:mobius}
For every $x\in\B{G,\D}$ and $N\in\D\cap\NS$,
\[
c_N(x)
=
\sum_{\substack{M\in\D\cap\NS\\ N\leq M}}
\mu(N,M)s_M(x).
\]
Thus the family of core coefficients and the family of cumulative core coefficients determine one another uniquely.
\end{prop}

\begin{proof}
By definition,
\[
s_N(x)
=
\sum_{\substack{M\in\D\cap\NS\\ N\leq M}}c_M(x).
\]
The assertion is the M\"obius inversion formula on the finite poset $\D\cap\NS$.
\end{proof}

\begin{cor}[A bound for core coefficients via the M\"obius function]
\label{cor:mobius-core-bound}
For every $u\in\B{G,\D}^{\times}$ and $N\in\D\cap\NS$,
\[
|c_N(u)|
\leq
\sum_{\substack{M\in\D\cap\NS\\ N\leq M}}
|\mu(N,M)|.
\]
\end{cor}

\begin{proof}
This follows from Proposition~\ref{prop:mobius}, Theorem~\ref{thm:core-cumulative}, and the triangle inequality.
\end{proof}

\begin{rem}
The right-hand side of Corollary~\ref{cor:mobius-core-bound} depends only on the interval
\[
\{M\in\D\cap\NS\mid N\leq M\}
\]
in the finite poset $\D\cap\NS$. In particular, when $\D=\operatorname{Sub}(G)$, bounds for core coefficients are controlled by the M\"obius function of the lattice of normal subgroups.
\end{rem}

Since $G$ has no normal subgroup properly containing $G$, for every $u\in\B{G,\D}^{\times}$,
\[
c_G(u)
=
s_G(u)
=
\varphi_G^{\D}(u)
\in\{-1,1\}.
\]

\begin{cor}[Parity of core coefficients at odd index]
\label{cor:odd-index-parity}
Let $u\in\B{G,\D}^{\times}$ and $N\in\D\cap\NS$ with $N<G$. If $|G:N|$ is odd, then
\[
c_N(u)\in2\Z.
\]
In particular, if $|G|$ is odd, then
\[
\mathfrak n_{G,\D}(u)
\equiv
1_{\B G}
\pmod{2\B{G,\D\cap\NS}}.
\]
\end{cor}

\begin{proof}
If $M\in\D\cap\NS$ and $N\leq M$, then $|G:M|$ is also odd. By Corollary~\ref{cor:zero-criterion},
\[
s_M(u)\in\{-1,1\},
\qquad
s_M(u)\equiv1\pmod2.
\]
Therefore Proposition~\ref{prop:mobius} gives
\[
c_N(u)
\equiv
\sum_{\substack{M\in\D\cap\NS\\N\leq M}}
\mu(N,M)
=
0
\pmod2.
\]
Since $c_G(u)\equiv1\pmod2$, the final assertion follows as well.
\end{proof}

\begin{prop}[The case of prime index]
\label{prop:prime-index}
Let $u\in\B{G,\D}^{\times}$ and $N\in\D\cap\NS$, and suppose that $|G:N|=p$ is prime. Then
\[
c_N(u)
=
\begin{cases}
0,
&p\text{ odd},\\[1mm]
\dfrac{\varphi_N^{\D}(u)-\varphi_G^{\D}(u)}{2},
&p=2.
\end{cases}
\]
In particular, if $p=2$, then
\[
c_N(u)\in\{-1,0,1\}.
\]
\end{prop}

\begin{proof}
The interval
\[
\{M\in\D\cap\NS\mid N\leq M\leq G\}
\]
is equal to $\{N,G\}$, so
\[
c_N(u)=s_N(u)-s_G(u).
\]
By Theorem~\ref{thm:core-cumulative},
\[
s_N(u)
=
\frac{\varphi_N^{\D}(u)+(p-1)\varphi_G^{\D}(u)}{p},
\]
and hence
\[
c_N(u)
=
\frac{\varphi_N^{\D}(u)-\varphi_G^{\D}(u)}{p}.
\]
The assertion for $p=2$ is immediate. Now suppose that $p$ is odd. The numerator is $0$ or $\pm2$, while the quotient is an integer. Hence the numerator must be divisible by $p$, and therefore it is $0$.
\end{proof}

\subsection{Compatibility with normal quotients}

\begin{prop}[Compatibility with normal quotients]\label{prop:quotient-compatibility}
Let $N\in\D\cap\NS$ and $x\in\B{G,\D}$. Then
\[
\mathfrak n_{G/N,\D_{\geq N}/N}(x^N)
=
\mathfrak n_{G,\D}(x)^N.
\]
Moreover, if $M\in\D\cap\NS$ and $N\leq M$, then
\[
c_{M/N}(x^N)
=
c_M(x),
\qquad
s_{M/N}(x^N)
=
s_M(x).
\]
\end{prop}

\begin{proof}
By~\eqref{eq:normal-core-containment}, the conditions $N\leq H$ and $N\leq\operatorname{core}_G(H)$ are equivalent. If $N\leq H$, then
\[
\operatorname{core}_{G/N}(H/N)
=
\operatorname{core}_G(H)/N.
\]

Consider an orbit-basis element $[G/H]$. If $N\nleq H$, then~\eqref{eq:normal-fixed-orbit} gives
\[
[G/H]^N
=
[G/\operatorname{core}_G(H)]^N
=
0.
\]
If $N\leq H$, then the same formula and the identity above give
\[
\begin{aligned}
\mathfrak n_{G/N,\D_{\geq N}/N}\bigl([G/H]^N\bigr)
&=
[(G/N)/(\operatorname{core}_G(H)/N)]\\
&=
[G/\operatorname{core}_G(H)]^N\\
&=
\mathfrak n_{G,\D}([G/H])^N.
\end{aligned}
\]
The first assertion follows by additivity.

Comparing coefficients in the normal orbit basis yields
\[
c_{M/N}(x^N)=c_M(x)
\]
for $N\leq M$. Summing over normal subgroups containing $M$ then gives
\[
s_{M/N}(x^N)=s_M(x).
\]
\end{proof}

\subsection{Examples and bounds on coefficients}

\begin{ex}[$C_2$]
Let $G=C_2$ and put
\[
u=1-[G/1].
\]
Then $u$ is a unit and
\[
c_G(u)=1,
\qquad c_1(u)=-1,
\]
whereas
\[
s_G(u)=1,
\qquad s_1(u)=0.
\]
\end{ex}

\begin{thm}[Core coefficients when the normal-subgroup lattice is a chain]
\label{thm:chain-normal-core-bound}
Suppose that the normal subgroups of a finite group $G$ form the chain
\[
1=N_0<N_1<\cdots<N_t=G.
\]
For every unit $u\in\B G^\times$,
\[
c_{N_i}(u)
=
\begin{cases}
s_{N_i}(u)-s_{N_{i+1}}(u),&0\leq i<t,\\
s_G(u),&i=t.
\end{cases}
\]
In particular,
\[
|c_N(u)|\leq2
\]
for every $N\trianglelefteq G$.
\end{thm}

\begin{proof}
By definition,
\[
s_{N_i}(u)
=
\sum_{j=i}^{t}c_{N_j}(u)
\]
for $0\leq i\leq t$. Hence
\[
c_{N_i}(u)
=
s_{N_i}(u)-s_{N_{i+1}}(u)
\]
for $0\leq i<t$, while
\[
c_G(u)=s_G(u).
\]
Theorem~\ref{thm:core-cumulative} gives $s_N(u)\in\{-1,0,1\}$, and the assertion follows.
\end{proof}

\begin{cor}[Core coefficients for symmetric groups]
\label{cor:symmetric-core-bound}
For every $n\geq2$, every unit $u\in\B{S_n}^{\times}$, and every normal subgroup $N\trianglelefteq S_n$,
\[
|c_N(u)|\leq2.
\]
\end{cor}

\begin{proof}
For $n=2,3,4$, the normal subgroups of $S_n$ are respectively
\[
1<S_2,
\qquad
1<A_3<S_3,
\qquad
1<V_4<A_4<S_4,
\]
and for $n\geq5$ they are
\[
1<A_n<S_n.
\]
Thus in every case the normal subgroups form a chain, so Theorem~\ref{thm:chain-normal-core-bound} applies.
\end{proof}

\begin{ex}[Units and core coefficients for $S_3$]
\label{ex:S3-core-coefficients}
Choose a subgroup of order $2$ in $S_3$ and denote it by $C_2$. Put
\[
u_{A_3}
:=
1_{\B{S_3}}-[S_3/A_3],
\qquad
u_{C_2}
:=
1_{\B{S_3}}-2[S_3/C_2]+[S_3/1].
\]
A direct calculation gives
\[
\B{S_3}^{\times}
=
\{
\pm1_{\B{S_3}},
\pm u_{A_3},
\pm u_{C_2},
\pm u_{A_3}u_{C_2}
\}.
\]
The normal subgroups of $S_3$ are $1$, $A_3$, and $S_3$, and the core coefficients and cumulative core coefficients of these units are as follows:
\[
\begin{array}{c|rrr|rrr}
u
&c_1(u)&c_{A_3}(u)&c_{S_3}(u)
&s_1(u)&s_{A_3}(u)&s_{S_3}(u)\\ \hline
1_{\B{S_3}}
&0&0&1&1&1&1\\
-1_{\B{S_3}}
&0&0&-1&-1&-1&-1\\
u_{A_3}
&0&-1&1&0&0&1\\
-u_{A_3}
&0&1&-1&0&0&-1\\
u_{C_2}
&-1&0&1&0&1&1\\
-u_{C_2}
&1&0&-1&0&-1&-1\\
u_{A_3}u_{C_2}
&-1&-1&1&-1&0&1\\
-u_{A_3}u_{C_2}
&1&1&-1&1&0&-1
\end{array}
\]
\end{ex}

\begin{ex}[Sharpness of the bound for $S_4$]
\label{ex:S4-sharp-core-bound}
Consider the subgroups
\[
\langle(12)\rangle,
\qquad
\langle(1234)\rangle,
\qquad
\operatorname{Stab}_{S_4}(4)\cong S_3
\]
of $S_4$, and put
\[
u_{S_4}
:=
1_{\B{S_4}}
+[S_4/\langle(12)\rangle]
-[S_4/\langle(1234)\rangle]
-2[S_4/\operatorname{Stab}_{S_4}(4)].
\]
A direct calculation shows that $u_{S_4}\in\B{S_4}^{\times}$ and
\[
\begin{array}{c|rrrr|rrrr}
u
&c_1(u)&c_{V_4}(u)&c_{A_4}(u)&c_{S_4}(u)
&s_1(u)&s_{V_4}(u)&s_{A_4}(u)&s_{S_4}(u)\\ \hline
u_{S_4}
&-2&0&0&1&-1&1&1&1
\end{array}
\]
Hence
\[
|c_1(u_{S_4})|=2,
\]
so the bound in Corollary~\ref{cor:symmetric-core-bound} is sharp.
\end{ex}

\begin{ex}[Core coefficients for elementary abelian $2$-groups]\label{ex:elementary-abelian}
Let $G=C_2^r$ with $r\geq1$, and let
\[
\mathcal H
:=
\{H\leq G\mid |G:H|=2\}
\]
be the set of all subgroups of index $2$. Define
\[
u_G
:=
\prod_{H\in\mathcal H}
\bigl(1_{\B G}-[G/H]\bigr)
\in\B G^\times.
\]

For $L\leq G$ and $H\in\mathcal H$,
\[
\varphi_L^G\bigl(1_{\B G}-[G/H]\bigr)
=
\begin{cases}
-1,&L\leq H,\\
1,&L\nleq H,
\end{cases}
\]
so
\[
\varphi_L^G(u_G)
=
(-1)^{\#\{H\in\mathcal H\mid L\leq H\}}.
\]

By the correspondence theorem, the subgroups of index $2$ in $G$ containing $L$ correspond bijectively to the subgroups of index $2$ in $G/L$. Such subgroups are precisely the kernels of nontrivial homomorphisms
\[
G/L\longrightarrow C_2.
\]
Since $\operatorname{Aut}(C_2)$ is trivial, distinct nontrivial homomorphisms have distinct kernels. If $G/L\cong C_2^k$, then
\[
\#\operatorname{Hom}(G/L,C_2)
=
2^k
=
|G:L|.
\]
Therefore
\[
\#\{H\in\mathcal H\mid L\leq H\}
=
|G:L|-1.
\]
This number is odd when $L<G$ and is $0$ when $L=G$. Hence
\[
\varphi_L^G(u_G)
=
\begin{cases}
1,&L=G,\\
-1,&L<G.
\end{cases}
\]

Now let $N\leq G$. By Theorem~\ref{thm:core-cumulative} and the above computation of marks,
\[
s_N(u_G)
=
\begin{cases}
1,&N=G,\\
0,&|G:N|=2,\\
-1,&|G:N|\geq4.
\end{cases}
\]
Indeed, when $|G:N|=2$, the marks appearing in the average are $\varphi_N^G(u_G)=-1$ and $\varphi_G^G(u_G)=1$. If $|G:N|\geq4$, then for every $g\in G$,
\[
|\langle N,g\rangle:N|\leq2<|G:N|,
\]
so $\langle N,g\rangle<G$, and all marks occurring in the average are $-1$.

For an elementary abelian $2$-group, the M\"obius function of the subgroup lattice is given, for $N\leq M\leq G$ with
\[
|M:N|=2^j,
\]
by
\[
\mu(N,M)
=
(-1)^j2^{\binom{j}{2}}.
\]

Suppose
\[
|G:N|=2^k.
\]
If $k=0$, that is, if $N=G$, then
\[
c_G(u_G)=1.
\]
Assume now that $k\geq1$. M\"obius inversion gives
\[
\begin{aligned}
c_N(u_G)
&=
\sum_{N\leq M\leq G}\mu(N,M)s_M(u_G)\\
&=
\mu(N,G)
-
\sum_{\substack{N\leq M\leq G\\ |G:M|\geq4}}
\mu(N,M).
\end{aligned}
\]
Since $N<G$,
\[
\sum_{N\leq M\leq G}\mu(N,M)=0.
\]
Thus
\[
c_N(u_G)
=
\sum_{\substack{N\leq M\leq G\\ |G:M|=2}}
\mu(N,M)
+
2\mu(N,G).
\]

There are $2^k-1$ subgroups of index $2$ in $G$ containing $N$, and for each such subgroup $M$ one has $|M:N|=2^{k-1}$. Therefore
\[
\begin{aligned}
c_N(u_G)
&=
(2^k-1)(-1)^{k-1}2^{\binom{k-1}{2}}
+
2(-1)^k2^{\binom{k}{2}}\\
&=
(-1)^k2^{\binom{k-1}{2}}.
\end{aligned}
\]
Hence
\[
c_N(u_G)
=
\begin{cases}
1,&N=G,\\
(-1)^k2^{\binom{k-1}{2}},
&|G:N|=2^k,\quad k\geq1.
\end{cases}
\]

In particular, for $N=1$,
\[
|c_1(u_G)|
=
2^{\binom{r-1}{2}}.
\]
By Theorem~\ref{thm:core-cumulative}, cumulative core coefficients of units always satisfy
\[
s_N(u)\in\{-1,0,1\}.
\]
By contrast, this example shows that the absolute values of individual core coefficients are unbounded among finite groups.
\end{ex}

\section{Refinement of cumulative coefficients by Yoshida characters}

In this section we work in the ordinary Burnside ring $\B G$. Since a partial Burnside ring is a subring of $\B G$ with the same identity element, each of its units is also a unit of $\B G$. Thus Yoshida's criterion applies to such units as well.

\subsection{Yoshida characters and cumulative core coefficients}

By Yoshida's unit criterion~\cite[Proposition~6.5]{Yoshida1990}, for every $u\in\B G^\times$ and $H\leq G$, the map
\[
N_G(H)/H\longrightarrow\{\pm1\},
\qquad
gH\longmapsto
\frac{\varphi_{\langle H,g\rangle}^{G}(u)}
     {\varphi_H^{G}(u)}
\]
is a group homomorphism. We refer to this homomorphism condition as Yoshida's criterion.

\begin{dfn}[Yoshida character associated with a normal core]
\label{dfn:yoshida-character}
For $N\trianglelefteq G$ and $u\in\B G^\times$, define
\[
\chi_N^u:G/N\longrightarrow\{\pm1\},
\qquad
\chi_N^u(gN):=
\frac{\varphi_{\langle N,g\rangle}^{G}(u)}{\varphi_N^{G}(u)}.
\]
\end{dfn}

Since $N_G(N)=G$, Yoshida's criterion shows that $\chi_N^u$ is a group homomorphism.

\begin{thm}[Cumulative core coefficients via Yoshida characters]
\label{thm:yoshida-dichotomy}
Let $u\in\B G^\times$ and $N\trianglelefteq G$. Then
\[
s_N(u)
=
\begin{cases}
\varphi_N^{G}(u),&\chi_N^u=1_{G/N},\\
0,&\chi_N^u\neq1_{G/N}.
\end{cases}
\]
\end{thm}

\begin{proof}
By Theorem~\ref{thm:core-cumulative} and the definition of $\chi_N^u$,
\[
s_N(u)
=
\frac{\varphi_N^{G}(u)}{|G/N|}
\sum_{q\in G/N}\chi_N^u(q).
\]
The average of a linear character is $1$ for the trivial character and $0$ for a nontrivial character.
\end{proof}

\begin{cor}[Range of cumulative core coefficients]
\label{cor:sN-range}
For $N\trianglelefteq G$,
\[
\{s_N(v)\mid v\in\B G^\times\}
=
\begin{cases}
\{-1,1\},&\operatorname{Hom}(G/N,C_2)=\{1_{G/N}\},\\
\{-1,0,1\},&\operatorname{Hom}(G/N,C_2)\neq\{1_{G/N}\}.
\end{cases}
\]
In particular, there exists a unit $v$ with $s_N(v)=0$ if and only if there is a normal subgroup of index $2$ containing $N$.
\end{cor}

\begin{proof}
Necessity follows from Theorem~\ref{thm:yoshida-dichotomy}. If there is a nontrivial homomorphism $\chi:G/N\to\{\pm1\}$, put $K/N=\ker\chi$. Then
\[
w_K:=1-[G/K]\in\B G^\times,
\qquad
\varphi_N^{G}(w_K)=-1,
\qquad
\chi_N^{w_K}=\chi,
\]
so $s_N(w_K)=0$. The values $1$ and $-1$ are realized by $1_{\B G}$ and $-1_{\B G}$, respectively.
\end{proof}

Taking $N=1$, we obtain in particular
\[
\operatorname{coeffsum}_G(u)
=
\begin{cases}
\varphi_1^{G}(u),&\chi_1^u=1_G,\\
0,&\chi_1^u\neq1_G.
\end{cases}
\]

\begin{cor}[M\"obius inversion via Yoshida characters]
\label{cor:yoshida-mobius}
Let $\mu_{\NS}$ denote the M\"obius function of the poset $\NS$ of normal subgroups. For $u\in\B G^\times$ and $N\trianglelefteq G$,
\[
c_N(u)
=
\sum_{\substack{M\trianglelefteq G,\ N\leq M\\\chi_M^u=1_{G/M}}}
\mu_{\NS}(N,M)\varphi_M^{G}(u).
\]
\end{cor}

\begin{proof}
Substitute Theorem~\ref{thm:yoshida-dichotomy} into Proposition~\ref{prop:mobius}.
\end{proof}

Theorem~\ref{thm:yoshida-dichotomy} shows that $s_N(u)$ is determined by the pair
\[
\bigl(\varphi_N^{G}(u),\chi_N^u\bigr).
\]
This pair is naturally compatible with multiplication of units. Indeed,
\[
\varphi_N^{G}(uv)=\varphi_N^{G}(u)\varphi_N^{G}(v),
\qquad
\chi_N^{uv}=\chi_N^u\chi_N^v,
\]
so
\[
\Psi_N:\B G^\times
\longrightarrow
\{\pm1\}\times\operatorname{Hom}(G/N,\{\pm1\}),
\qquad
u\longmapsto\bigl(\varphi_N^{G}(u),\chi_N^u\bigr)
\]
is a group homomorphism.

\begin{cor}[Cumulative core coefficients of products]
\label{cor:yoshida-product}
For $u,v\in\B G^\times$ and $N\trianglelefteq G$,
\[
s_N(uv)
=
\begin{cases}
\varphi_N^{G}(u)\varphi_N^{G}(v),&\chi_N^u=\chi_N^v,\\
0,&\chi_N^u\neq\chi_N^v.
\end{cases}
\]
\end{cor}

\begin{proof}
Since $\Psi_N$ is a group homomorphism and every linear character has order at most $2$, the conditions $\chi_N^{uv}=1_{G/N}$ and $\chi_N^u=\chi_N^v$ are equivalent. Apply Theorem~\ref{thm:yoshida-dichotomy} to $uv$.
\end{proof}

\subsection{A homomorphism refinement of core compression}

Collect the homomorphisms $\Psi_N$ for all $N\trianglelefteq G$ into
\[
\Psi:\B G^\times\longrightarrow
\prod_{N\in\NS}
\bigl(\{\pm1\}\times\operatorname{Hom}(G/N,\{\pm1\})\bigr),
\]
\[
\Psi(u):=\bigl(\Psi_N(u)\bigr)_{N\in\NS}.
\]

For each $N\trianglelefteq G$, define
\[
\operatorname{Av}_N:\{\pm1\}\times\operatorname{Hom}(G/N,\{\pm1\})
\longrightarrow\{-1,0,1\}
\]
by
\[
\operatorname{Av}_N(\varepsilon,\chi)
:=
\frac{\varepsilon}{|G/N|}
\sum_{gN\in G/N}\chi(gN).
\]
If $\chi$ is nontrivial, then its kernel has index $2$ in $G/N$, and therefore
\[
\operatorname{Av}_N(\varepsilon,\chi)
=
\begin{cases}
\varepsilon,&\chi=1_{G/N},\\
0,&\chi\neq1_{G/N}.
\end{cases}
\]
Combining these maps, define
\[
\operatorname{Av}:
\prod_{N\in\NS}
\bigl(\{\pm1\}\times\operatorname{Hom}(G/N,\{\pm1\})\bigr)
\longrightarrow
\prod_{N\in\NS}\{-1,0,1\}
\]
by
\[
\operatorname{Av}\bigl((\varepsilon_N,\chi_N)_{N\in\NS}\bigr)
:=
\bigl(\operatorname{Av}_N(\varepsilon_N,\chi_N)\bigr)_{N\in\NS}.
\]
By Theorem~\ref{thm:yoshida-dichotomy},
\[
\operatorname{Av}(\Psi(u))
=
\bigl(s_N(u)\bigr)_{N\in\NS}.
\]
Thus, on units, core compression factors as
\[
\begin{aligned}
\B G^\times
&\xrightarrow{\ \Psi\ }
\prod_{N\in\NS}
\bigl(\{\pm1\}\times\operatorname{Hom}(G/N,\{\pm1\})\bigr)\\
&\xrightarrow{\ \operatorname{Av}\ }
\prod_{N\in\NS}\{-1,0,1\}
\xrightarrow{\ \text{M\"obius inversion}\ }
\B{G,\NS}.
\end{aligned}
\]
The first map is a group homomorphism. On the other hand, for every nontrivial $\chi$, the map $\operatorname{Av}_N$ sends both
\[
(1,\chi),\quad(-1,\chi)
\]
to the same value $0$. Thus the character information retained by $\Psi_N$ is lost at the averaging stage.

\begin{prop}[Surjectivity of $\Psi_N$]
\label{prop:psi-coordinate-surjective}
For every $N\trianglelefteq G$,
\[
\Psi_N:
\B G^\times
\longrightarrow
\{\pm1\}\times\operatorname{Hom}(G/N,\{\pm1\})
\]
is surjective.
\end{prop}

\begin{proof}
The elements $(1,1_{G/N})$ and $(-1,1_{G/N})$ are realized by $1_{\B G}$ and $-1_{\B G}$, respectively. For a nontrivial homomorphism $\chi:G/N\to\{\pm1\}$, put $K/N=\ker\chi$. Then
\[
\Psi_N(1-[G/K])=(-1,\chi),
\qquad
\Psi_N(-1+[G/K])=(1,\chi).
\]
\end{proof}

\begin{cor}[A bound on the image of core compression]
\label{cor:core-image-bound}
Put
\[
\operatorname{NS}_{\mathrm{triv}}(G)
:=
\{N\trianglelefteq G\mid\operatorname{Hom}(G/N,C_2)=\{1_{G/N}\}\},
\]
\[
\operatorname{NS}_{\mathrm{nontriv}}(G)
:=
\{N\trianglelefteq G\mid\operatorname{Hom}(G/N,C_2)\neq\{1_{G/N}\}\}.
\]
Then
\[
\#\mathfrak n_G(\B G^\times)
\leq
2^{\#\operatorname{NS}_{\mathrm{triv}}(G)}
3^{\#\operatorname{NS}_{\mathrm{nontriv}}(G)}.
\]
\end{cor}

\begin{proof}
By Corollary~\ref{cor:sN-range}, the cumulative core coefficient $s_N(u)$ belongs to $\{\pm1\}$ if $N\in\operatorname{NS}_{\mathrm{triv}}(G)$ and to $\{-1,0,1\}$ if $N\in\operatorname{NS}_{\mathrm{nontriv}}(G)$. Since M\"obius inversion is bijective, multiplying the number of possibilities in each coordinate gives the result.
\end{proof}

\begin{cor}[Yoshida characters and the parity of core coefficients]
\label{cor:no-c2-parity}
For every $u\in\B G^\times$ and $N\trianglelefteq G$,
\[
c_N(u)
\equiv
\sum_{\substack{M\geq N\\\chi_M^u=1_{G/M}}}
\mu_{\NS}(N,M)
\pmod2.
\]
Thus the parity of a core coefficient is determined solely by the pattern of trivial and nontrivial Yoshida characters.

In particular, if $N<G$ and $\operatorname{Hom}(G/N,C_2)=\{1_{G/N}\}$, then
\[
c_N(u)\in2\Z.
\]
\end{cor}

\begin{proof}
Reducing each mark in Corollary~\ref{cor:yoshida-mobius} modulo $2$ gives the first congruence.

Now assume $\operatorname{Hom}(G/N,C_2)=\{1_{G/N}\}$. If $M\geq N$, the natural surjection $G/N\to G/M$ induces an injection
\[
\operatorname{Hom}(G/M,C_2)
\hookrightarrow
\operatorname{Hom}(G/N,C_2).
\]
Hence $\operatorname{Hom}(G/M,C_2)=\{1_{G/M}\}$, so $\chi_M^u=1_{G/M}$. Therefore, for $N<G$,
\[
c_N(u)
\equiv
\sum_{M\geq N}\mu_{\NS}(N,M)
=
0
\pmod2.
\]
\end{proof}

\section{Applications to the normal partial Burnside ring}

When it is necessary to specify the ambient group in which core coefficients are taken, we write
\[
c_N^G(x),\qquad s_N^G(x).
\]
When the group is clear from the context, we continue to write simply
\[
c_N(x),\qquad s_N(x).
\]

\subsection{Reduction of the core-coefficient problem to normal quotients}

Let $N\trianglelefteq G$. Viewing a $G/N$-set as a $G$-set through the quotient map $G\to G/N$ gives the inflation homomorphism
\[
\operatorname{Inf}_{G/N}^{G}:\B{G/N}\longrightarrow\B G.
\]
It is a unital ring homomorphism, and on the orbit basis
\[
\operatorname{Inf}_{G/N}^{G}
\bigl([(G/N)/(H/N)]\bigr)
=[G/H]
\qquad (N\leq H\leq G).
\]
For the biset-functor structure of the Burnside functor, including inflation, see for example~\cite{Bouc2010}.

The fixed-point map $(-)^N$ defined in Section~3 satisfies
\[
(-)^N\circ\operatorname{Inf}_{G/N}^{G}
=\operatorname{id}_{\B{G/N}}.
\]
Hence its restriction to unit groups,
\[
(-)^N:\B G^\times\longrightarrow\B{G/N}^\times,
\]
is surjective.

For finite groups $G_1,G_2$, the map obtained by extending linearly the correspondence
\[
[G_1/H_1]\otimes [G_2/H_2]
\longmapsto
[(G_1\times G_2)/(H_1\times H_2)]
\]
defines a ring homomorphism
\[
\operatorname{Prod}:\B{G_1}\otimes_{\mathbb Z}\B{G_2}
\longrightarrow\B{G_1\times G_2}.
\]

\begin{prop}[Compatibility of inflation and core compression]\label{prop:inflation-core}
Let $N\trianglelefteq G$ and $x\in\B{G/N}$. Then
\[
\mathfrak n_G\left(\operatorname{Inf}_{G/N}^{G}(x)\right)
=\operatorname{Inf}_{G/N}^{G}\left(\mathfrak n_{G/N}(x)\right).
\]
More explicitly, for every $M\trianglelefteq G$,
\[
c_M^G\left(\operatorname{Inf}_{G/N}^{G}(x)\right)=
\begin{cases}
c_{M/N}^{G/N}(x),&N\leq M,\\[4pt]
0,&N\nleq M.
\end{cases}
\]
\end{prop}

\begin{proof}
On the orbit basis,
\[
\operatorname{Inf}_{G/N}^{G}
\bigl([(G/N)/(H/N)]\bigr)
=[G/H]
\qquad(N\leq H\leq G).
\]
Since $N\trianglelefteq G$ and $N\leq H$, we have $N\leq\operatorname{core}_G(H)$ and
\[
\operatorname{core}_{G/N}(H/N)=\operatorname{core}_G(H)/N.
\]
Thus, for $M\trianglelefteq G$ with $N\leq M$,
\[
\operatorname{core}_G(H)=M
\quad\Longleftrightarrow\quad
\operatorname{core}_{G/N}(H/N)=M/N.
\]
Therefore
\[
c_M^G\left(\operatorname{Inf}_{G/N}^{G}(x)\right)
=c_{M/N}^{G/N}(x).
\]
On the other hand, every orbit $[G/H]$ occurring in an inflated element satisfies $N\leq\operatorname{core}_G(H)$, so if $N\nleq M$, no orbit with core $M$ occurs. Hence
\[
c_M^G\left(\operatorname{Inf}_{G/N}^{G}(x)\right)=0.
\]
This proves the assertion.
\end{proof}

\begin{cor}[Reduction of core coefficients to normal quotients]\label{cor:quotient-value-set}
If $N\trianglelefteq G$, then
\[
\{c_N^G(u)\mid u\in\B G^\times\}
=
\{c_1^{G/N}(v)\mid v\in\B{G/N}^\times\}.
\]
In particular,
\[
c_N^G(u)=0
\qquad
\text{for every }u\in\B G^\times
\]
if and only if
\[
c_1^{G/N}(v)=0
\qquad
\text{for every }v\in\B{G/N}^\times.
\]
\end{cor}

\begin{proof}
By Proposition~\ref{prop:quotient-compatibility}, $c_N^G(u)=c_1^{G/N}(u^N)$. The fixed-point map is surjective on unit groups because inflation is a right inverse.
\end{proof}

\begin{prop}[Compatibility of direct products and core compression]\label{prop:external-product-core}
Let $G_1,G_2$ be finite groups and $x_i\in\B{G_i}$. Then
\[
\mathfrak n_{G_1\times G_2}\bigl(\operatorname{Prod}(x_1\otimes x_2)\bigr)
=\operatorname{Prod}\bigl(\mathfrak n_{G_1}(x_1)\otimes \mathfrak n_{G_2}(x_2)\bigr).
\]
Consequently, for $N_i\trianglelefteq G_i$,
\[
c_{N_1\times N_2}\bigl(\operatorname{Prod}(x_1\otimes x_2)\bigr)
=c_{N_1}(x_1)c_{N_2}(x_2),
\]
and
\[
s_{N_1\times N_2}\bigl(\operatorname{Prod}(x_1\otimes x_2)\bigr)
=s_{N_1}(x_1)s_{N_2}(x_2).
\]
\end{prop}

\begin{proof}
Use
\[
\operatorname{core}_{G_1\times G_2}(H_1\times H_2)
=\operatorname{core}_{G_1}(H_1)\times\operatorname{core}_{G_2}(H_2).
\]
The product formula for core coefficients follows by expansion. The corresponding formula for cumulative coefficients follows either from the zeta transform or from the direct-product formula for fixed-point sets.
\end{proof}

\begin{rem}
In particular, whether $c_N^G(u)=0$ for every unit $u\in\B G^\times$ depends only on the isomorphism type of the quotient group $G/N$.
\end{rem}

\subsection{Core-trivial units}

\begin{dfn}[Core-cyclic subgroup]
A subgroup $H\leq G$ is called \emph{core-cyclic} if
\[
H/\operatorname{core}_G(H)
\]
is cyclic.
\end{dfn}

\begin{thm}[The group of core-trivial units]\label{thm:core-trivial-subgroup}
Put
\[
\mathcal U_{\mathrm{core}}(G)
:=\{u\in\B G^\times\mid \mathfrak n_G(u)=1_{\B G}\},
\]
and call its elements \emph{core-trivial units}. Then
\[
\mathcal U_{\mathrm{core}}(G)
=\ker\Psi
=\left\{u\in\B G^\times\ \middle|\
\varphi_H^{G}(u)=1\quad\text{for every core-cyclic subgroup }H\leq G\right\}.
\]
\end{thm}

\begin{proof}
By the invertibility of M\"obius inversion and Theorem~\ref{thm:yoshida-dichotomy},
\[
\mathfrak n_G(u)=1_{\B G}
\Longleftrightarrow
s_N(u)=1\ \text{for all }N\trianglelefteq G
\Longleftrightarrow
\varphi_N^{G}(u)=1,\ \chi_N^u=1_{G/N}\ \text{for all }N\trianglelefteq G.
\]
Thus the first two sets coincide.

Suppose that $H$ is core-cyclic and put $N=\operatorname{core}_G(H)$. Then $H=\langle N,g\rangle$ for some $g\in H$, and hence $u\in\ker\Psi$ implies $\varphi_H^G(u)=1$. Conversely, suppose all marks at core-cyclic subgroups are equal to $1$. Every normal subgroup $N$ and every subgroup $\langle N,g\rangle$ are core-cyclic, so $\varphi_N^G(u)=1$ and $\chi_N^u=1_{G/N}$ for every $N\trianglelefteq G$.
\end{proof}

\begin{cor}[Units whose core compression is $-1_{\B G}$]
The set $\mathcal U_{\mathrm{core}}(G)$ is a subgroup of $\B G^\times$, and
\[
\{u\in\B G^\times\mid \mathfrak n_G(u)=-1_{\B G}\}
=-\mathcal U_{\mathrm{core}}(G).
\]
\end{cor}

\begin{proof}
The subgroup assertion follows from the description $\mathcal U_{\mathrm{core}}(G)=\ker\Psi$ in Theorem~\ref{thm:core-trivial-subgroup}. Since $\mathfrak n_G(-u)=-\mathfrak n_G(u)$, the units with core compression $-1_{\B G}$ are precisely the elements of $-\mathcal U_{\mathrm{core}}(G)$.
\end{proof}

Let $R_{\mathbb Q}(G)$ denote the representation ring of finite-dimensional rational representations of $G$.
For a finite $G$-set $X$, let $\mathbb QX$ be the rational permutation module with basis $X$. This assignment induces the unital ring homomorphism
\[
\ell_G:\B G\longrightarrow R_{\mathbb Q}(G),
\qquad
[X]\longmapsto[\mathbb QX],
\]
called the \emph{linearization map}. For the linearization of the Burnside functor, see for example~\cite{Bouc2010}.

For every $x\in\B G$ and $g\in G$, the character of the linearized representation satisfies
\[
\chi_{\ell_G(x)}(g)
=\varphi_{\langle g\rangle}^{G}(x).
\]
Indeed, for a finite $G$-set $X$, the value at $g$ of its permutation character is $\#X^{\langle g\rangle}$.

\begin{cor}[Linearization and core-trivial units]\label{cor:linearization-kernel}
\[
\mathcal U_{\mathrm{core}}(G)
\subseteq
\left\{u\in\B G^\times\ \middle|\
\ell_G(u)=1_{R_{\mathbb Q}(G)}\right\}.
\]
\end{cor}

\begin{proof}
Every cyclic subgroup is core-cyclic. Thus Theorem~\ref{thm:core-trivial-subgroup} gives $\varphi_{\langle g\rangle}^{G}(u)=1$ for every $g\in G$. By the character formula above, the character of $\ell_G(u)$ is the trivial character, so $\ell_G(u)=1_{R_{\mathbb Q}(G)}$.
\end{proof}

\begin{prop}[Criterion for equal core compressions]\label{prop:same-core-compression}
Let $u,v\in\B G^\times$. Then $\mathfrak n_G(u)=\mathfrak n_G(v)$ if and only if, for every $N\trianglelefteq G$, one of the following holds:
\begin{enumerate}
 \item both $\chi_N^u$ and $\chi_N^v$ are nontrivial;
 \item $\chi_N^u=\chi_N^v=1_{G/N}$ and $\varphi_N^{G}(u)=\varphi_N^{G}(v)$.
\end{enumerate}
\end{prop}

\begin{proof}
The equality $\mathfrak n_G(u)=\mathfrak n_G(v)$ is equivalent to $s_N(u)=s_N(v)$ for every $N$. The assertion then follows from Theorem~\ref{thm:yoshida-dichotomy}.
\end{proof}

\subsection{Nonzero core coefficients in the normal partial Burnside ring}

Consider the normal partial Burnside ring $\B{G,\NS}$. By the known generating theorem recalled in Section~2, its unit group is generated by $-1_{\B G}$ and the elements
\[
w_{\lambda}:=1_{\B G}-[G/\ker\lambda]
\qquad
(\lambda\in\operatorname{Hom}(G,C_2)\setminus\{1_G\}).
\]
We now classify the nonzero core coefficients arising from these generators.

\begin{thm}[Classification of nonzero core coefficients]\label{thm:normal-pbr-nonzero-core}
Let $N\trianglelefteq G$. The following are equivalent:
\begin{enumerate}
 \item $c_N(u)\neq0$ for some $u\in\B{G,\NS}^{\times}$;
 \item $G/N$ is an elementary abelian $2$-group.
\end{enumerate}
Consequently,
\[
\bigcup_{u\in\B{G,\NS}^{\times}}
\{N\trianglelefteq G\mid c_N(u)\neq0\}
=
\{N\trianglelefteq G\mid G/N\text{ is an elementary abelian }2\text{-group}\}.
\]
The trivial quotient $G/G$ is regarded as an elementary abelian $2$-group of rank $0$.
\end{thm}

\begin{proof}
Every unit can be written as
\[
u=\varepsilon\prod_{\lambda\in S}w_{\lambda},
\qquad \varepsilon\in\{\pm1\},
\quad S\subseteq \operatorname{Hom}(G,C_2)\setminus\{1_G\}.
\]
Every stabilizer appearing in the expansion is of the form
\[
N_T:=\bigcap_{\lambda\in T}\ker\lambda
\qquad(T\subseteq S).
\]
Thus, if $c_N(u)\neq0$, then $N=N_T$ for some $T$, and $G/N$ embeds into $C_2^T$. Hence $G/N$ is an elementary abelian $2$-group.

Conversely, suppose $G/N\cong C_2^r$. Choose a basis $\lambda_1,\dots,\lambda_r$ for this character group and inflate it to $G$. In the expansion of
\[
u=\prod_{i=1}^r(1-[G/\ker\lambda_i]),
\]
the only term whose stabilizer is $N$ is obtained by choosing all $r$ factors, and its coefficient is $(-1)^r$. Hence $c_N(u)=(-1)^r\neq0$.
\end{proof}

\begin{dfn}
A finite group $Q$ is called a \emph{Dedekind group} if every subgroup of $Q$ is normal.
\end{dfn}

The structure of finite Dedekind groups is well known. By the Dedekind--Baer theorem, a finite Dedekind group $Q$ is either abelian or has the form
\[
Q\cong Q_8\times E\times A,
\]
where $Q_8$ is the quaternion group of order $8$, $E$ is an elementary abelian $2$-group, and $A$ is an abelian group of odd order; see, for example,~\cite[Theorem~5.3.7]{Robinson1996}.

\begin{cor}[Nonzero core coefficients for Dedekind quotients]
\label{cor:dedekind-quotient-nonzero-core}
Let $N\trianglelefteq G$ and suppose that $G/N$ is a Dedekind group. Then the following are equivalent:
\begin{enumerate}
 \item $c_N^G(u)\neq0$ for some $u\in\B G^\times$;
 \item $G/N$ is an elementary abelian $2$-group.
\end{enumerate}
\end{cor}

\begin{proof}
By Corollary~\ref{cor:quotient-value-set},
\[
\{c_N^G(u)\mid u\in\B G^\times\}
=
\{c_1^{G/N}(v)\mid v\in\B{G/N}^\times\}.
\]
Since $G/N$ is a Dedekind group,
\[
\B{G/N}=\B{G/N,\operatorname{NS}(G/N)}.
\]
The result follows by applying Theorem~\ref{thm:normal-pbr-nonzero-core} to the trivial subgroup of $G/N$.
\end{proof}

\subsection{Cumulative core coefficients of units in the normal PBR}

For $S\subseteq\operatorname{Hom}(G,C_2)\setminus\{1_G\}$ and $\varepsilon\in\{\pm1\}$, consider the unit
\[
u_{\varepsilon,S}
:=\varepsilon\prod_{\lambda\in S}(1_{\B G}-[G/\ker\lambda])
\in\B{G,\NS}^{\times}.
\]

\begin{thm}[Standard form for units in the normal PBR]\label{thm:normal-unit-unique-subset}
Every $u\in\B{G,\NS}^{\times}$ has a unique expression
\[
u=u_{\varepsilon,S}
=\varepsilon\prod_{\lambda\in S}(1_{\B G}-[G/\ker\lambda]).
\]
Consequently,
\[
\B{G,\NS}^{\times}
\cong C_2\times C_2^{\,\#(\operatorname{Hom}(G,C_2)\setminus\{1_G\})}.
\]
\end{thm}

\begin{proof}
Existence follows from~\cite[Proposition~4.6]{Wakatake2018}. We prove uniqueness. First,
\[
\varphi_G^{G}(u_{\varepsilon,S})=\varepsilon,
\]
so $\varepsilon$ is uniquely determined by $u_{\varepsilon,S}$. Next, for any $\lambda\in\operatorname{Hom}(G,C_2)\setminus\{1_G\}$,
\[
\varphi_{\ker\lambda}^{G}(u_{\varepsilon,S})
=
\begin{cases}
-\varepsilon,&\lambda\in S,\\
\varepsilon,&\lambda\notin S.
\end{cases}
\]
Hence
\[
\lambda\in S
\quad\Longleftrightarrow\quad
\varphi_{\ker\lambda}^{G}(u_{\varepsilon,S})
=-\varphi_G^{G}(u_{\varepsilon,S}),
\]
so $S$ is uniquely determined as well.
\end{proof}

\begin{prop}[Cumulative core coefficients of units in the normal PBR]\label{prop:normal-unit-cumulative}
Let $N\trianglelefteq G$. Here $\chi_N^{u_{\varepsilon,S}}$ is the Yoshida character from Definition~\ref{dfn:yoshida-character}. Each $\lambda$ satisfying $N\leq\ker\lambda$ induces a character of $G/N$, so
\[
gN\longmapsto
\prod_{\substack{\lambda\in S\\ N\leq\ker\lambda}}\lambda(g)
\]
is a well-defined character of $G/N$. Then
\[
\varphi_N^{G}(u_{\varepsilon,S})
=\varepsilon(-1)^{\#\{\lambda\in S\mid N\leq\ker\lambda\}},
\]
\[
\chi_N^{u_{\varepsilon,S}}(gN)
=\prod_{\substack{\lambda\in S\\ N\leq\ker\lambda}}\lambda(g),
\]
and
\[
s_N(u_{\varepsilon,S})=
\begin{cases}
\varepsilon(-1)^{\#\{\lambda\in S\mid N\leq\ker\lambda\}},
&\displaystyle\prod_{\substack{\lambda\in S\\ N\leq\ker\lambda}}\lambda=1_G,\\
0,&\displaystyle\prod_{\substack{\lambda\in S\\ N\leq\ker\lambda}}\lambda\neq1_G.
\end{cases}
\]
\end{prop}

\begin{proof}
The $H$-mark of $1_{\B G}-[G/\ker\lambda]$ is $-1$ if $H\leq\ker\lambda$ and $1$ otherwise. Multiplying the contributions of the factors gives the first formula. Taking $H=\langle N,g\rangle$ and forming $\varphi_H^G(u_{\varepsilon,S})/\varphi_N^G(u_{\varepsilon,S})$ gives the formula for the Yoshida character. The formula for the cumulative core coefficient then follows from Theorem~\ref{thm:yoshida-dichotomy}.
\end{proof}

In particular,
\[
s_N(u_{\varepsilon,S})\neq0
\quad\Longleftrightarrow\quad
\prod_{\substack{\lambda\in S\\ N\leq\ker\lambda}}\lambda=1_G.
\]
When it is nonzero,
\[
s_N(u_{\varepsilon,S})
=\varepsilon(-1)^{\#\{\lambda\in S\mid N\leq\ker\lambda\}}.
\]

\section{Linear criteria and decompositions for core-trivial units}
\subsection{Core-trivial units and primitive idempotents}

In this section we reduce the nontriviality of $\mathcal U_{\mathrm{core}}(G)$ to an integrality condition for sums of rational primitive idempotents corresponding to conjugacy classes of non-core-cyclic subgroups.

Write
\[
\mathbb Q\B G
:=
\mathbb Q\otimes_{\mathbb Z}\B G
\]
for the rationalization of the Burnside ring. Extending scalars in the mark homomorphism defined in Section~2 yields a $\mathbb Q$-algebra isomorphism
\[
\varphi_{\mathbb Q}^G:
\mathbb Q\B G
\longrightarrow
\prod_{(H)\in\operatorname{Sub}(G)^c}\mathbb Q;
\]
see, for example,~\cite[\S 3.3]{Bouc2010}. Hence, for each conjugacy class $(H)$ of subgroups, there is a unique element $e_H^G\in\mathbb Q\B G$ satisfying
\[
\varphi_K^{G}(e_H^G)
=
\begin{cases}
1,&(K)=(H),\\
0,&(K)\neq(H),
\end{cases}
\]
and these elements are precisely the primitive idempotents of $\mathbb Q\B G$. In particular,
\[
1_{\B G}=\sum_{(H)\in\operatorname{Sub}(G)^c}e_H^G,
\qquad
e_H^Ge_K^G=
\begin{cases}
e_H^G,&(H)=(K),\\
0,&(H)\neq(K).
\end{cases}
\]

Let
\[
\mu_{\operatorname{Sub}(G)}(K,H)
\qquad(K\leq H\leq G)
\]
denote the M\"obius function of the subgroup lattice of $G$. By the Gluck--Yoshida formula for primitive idempotents~\cite{Gluck1981,Yoshida1983},
\[
e_H^G
=
\frac{1}{|N_G(H)|}
\sum_{K\leq H}
|K|\,\mu_{\operatorname{Sub}(G)}(K,H)[G/K].
\]
See also~\cite[Theorem~3.3.5]{Bouc2010}. We write
\[
\operatorname{Sub}_{\mathrm{ncc}}(G)
:=
\left\{H\leq G\ \middle|\ H/\operatorname{core}_G(H)\text{ is noncyclic}\right\}
\]
for the set of non-core-cyclic subgroups of $G$. Here ``ncc'' stands for ``non-core-cyclic''. The set of their $G$-conjugacy classes is denoted by $\operatorname{Sub}_{\mathrm{ncc}}(G)^c$.

For $S,T\subseteq\operatorname{Sub}(G)^c$, write
\[
S\triangle T
:=
(S\setminus T)\cup(T\setminus S)
\]
for their symmetric difference. For $S\subseteq\operatorname{Sub}(G)^c$, put
\[
e_S:=\sum_{(H)\in S}e_H^G.
\]

\begin{lem}\label{lem:idempotent-symmetric-difference}
For every subgroup $K\leq G$,
\[
\varphi_K^G(e_S)
=
\begin{cases}
1,&(K)\in S,\\
0,&(K)\notin S.
\end{cases}
\]
Moreover, for all $S,T\subseteq\operatorname{Sub}(G)^c$,
\[
e_Se_T=e_{S\cap T},
\]
and
\[
(1_{\B G}-2e_S)(1_{\B G}-2e_T)
=
1_{\B G}-2e_{S\triangle T}.
\]
\end{lem}

\begin{proof}
By the definition of the primitive idempotents,
\[
\varphi_K^G(e_H^G)
=
\begin{cases}
1,&(K)=(H),\\
0,&(K)\neq(H),
\end{cases}
\]
which immediately gives the first formula. Orthogonality of the primitive idempotents yields
\[
e_Se_T
=
\sum_{(H)\in S}\sum_{(L)\in T}e_H^Ge_L^G
=
e_{S\cap T}.
\]
Therefore
\[
\begin{aligned}
(1_{\B G}-2e_S)(1_{\B G}-2e_T)
&=1_{\B G}-2e_S-2e_T+4e_{S\cap T}\\
&=1_{\B G}-2e_{S\triangle T}.
\end{aligned}
\]
\end{proof}

The primitive-idempotent decomposition above and the description of Burnside-ring units~\cite{Yoshida1990} give the bijection
\[
u\longmapsto\frac{1_{\B G}-u}{2}
\]
between $\B G^\times$ and
\[
\left\{
e\in\mathbb Q\B G
\ \middle|\
e^2=e,\quad 2e\in\B G
\right\}.
\]
Its inverse is
\[
e\longmapsto1_{\B G}-2e.
\]
Consequently,
\[
\B G^\times
=
\left\{
1_{\B G}-2e_S
\ \middle|\
S\subseteq\operatorname{Sub}(G)^c,\quad 2e_S\in\B G
\right\}.
\]

\begin{thm}[Core-trivial units and primitive idempotents]\label{thm:core-trivial-idempotents}
Let
\[
\mathcal S_2(G)
:=
\left\{
S\subseteq\operatorname{Sub}_{\mathrm{ncc}}(G)^c
\ \middle|\
2e_S\in\B G
\right\}.
\]
Then
\[
\mathcal U_{\mathrm{core}}(G)
=
\left\{
1_{\B G}-2e_S
\ \middle|\
S\in\mathcal S_2(G)
\right\}.
\]
Moreover, $\mathcal S_2(G)$ is an $\F_2$-vector space with symmetric difference as addition.
\end{thm}

\begin{proof}
By the description above, every $u\in\B G^\times$ can be written uniquely as
\[
u=1_{\B G}-2e_S,
\]
where
\[
S\subseteq\operatorname{Sub}(G)^c,
\qquad
2e_S\in\B G.
\]
By Lemma~\ref{lem:idempotent-symmetric-difference},
\[
\varphi_H^G(u)
=
\begin{cases}
-1,&(H)\in S,\\
1,&(H)\notin S.
\end{cases}
\]
By Theorem~\ref{thm:core-trivial-subgroup}, the unit $u$ is core-trivial if and only if
\[
\varphi_H^G(u)=1
\]
for every core-cyclic subgroup $H\leq G$. This is equivalent to
\[
S\subseteq\operatorname{Sub}_{\mathrm{ncc}}(G)^c.
\]
Hence
\[
\mathcal U_{\mathrm{core}}(G)
=
\left\{
1_{\B G}-2e_S
\ \middle|\
S\in\mathcal S_2(G)
\right\}.
\]

If $S,T\in\mathcal S_2(G)$, Lemma~\ref{lem:idempotent-symmetric-difference} gives
\[
(1_{\B G}-2e_S)(1_{\B G}-2e_T)
=
1_{\B G}-2e_{S\triangle T}.
\]
The left-hand side lies in $\B G$, so $2e_{S\triangle T}\in\B G$, and hence $S\triangle T\in\mathcal S_2(G)$. Since $e_{\varnothing}=0$, we also have $\varnothing\in\mathcal S_2(G)$. Thus $\mathcal S_2(G)$ is an $\F_2$-vector space under symmetric difference.
\end{proof}

To each subset $S\subseteq\operatorname{Sub}_{\mathrm{ncc}}(G)^c$, associate its characteristic function
\[
x^S\in\F_2^{\operatorname{Sub}_{\mathrm{ncc}}(G)^c},
\qquad
x^S_{(H)}=
\begin{cases}
1,&(H)\in S,\\
0,&(H)\notin S.
\end{cases}
\]
Whenever $\mathcal S_2(G)$ is used in a linear-algebraic context, it denotes the image
\[
\{x^S\mid S\in\mathcal S_2(G)\}
\subseteq\F_2^{\operatorname{Sub}_{\mathrm{ncc}}(G)^c}
\]
under this correspondence.

\begin{cor}[Criterion via the table of marks]\label{cor:core-mark-matrix}
Let
\[
M=
\bigl(\varphi_H^G([G/K])\bigr)_{(H),(K)\in\operatorname{Sub}(G)^c}
\]
be the table of marks. For $x\in\F_2^{\operatorname{Sub}_{\mathrm{ncc}}(G)^c}$, define
\[
\widetilde x_{(H)}
=
\begin{cases}
x_{(H)},&(H)\in\operatorname{Sub}_{\mathrm{ncc}}(G)^c,\\
0,&(H)\notin\operatorname{Sub}_{\mathrm{ncc}}(G)^c,
\end{cases}
\]
so that
\[
\widetilde x\in\{0,1\}^{\operatorname{Sub}(G)^c}.
\]
Then
\[
x\in\mathcal S_2(G)
\iff
2M^{-1}\widetilde x\in\mathbb Z^{\operatorname{Sub}(G)^c}.
\]
\end{cor}

\begin{proof}
The vector $M^{-1}\widetilde x$ is the coefficient vector, with respect to the orbit basis, of $e_S$, where $S$ is the subset corresponding to $x$. Thus the assertion is equivalent to $2e_S\in\B G$, which is exactly the defining condition.
\end{proof}

For $S=\{(H)\}$, the primitive-idempotent formula gives the following integrality criterion.

\begin{prop}[Criterion for a single subgroup class]\label{prop:single-subgroup-core-unit}
Let $H\leq G$ be non-core-cyclic. For each subgroup conjugacy class $(K)$, put
\[
a_K(H)
:=
\sum_{\substack{L\leq H\\(L)=(K)}}
|L|\,\mu_{\operatorname{Sub}(G)}(L,H).
\]
Then
\[
1_{\B G}-2e_H^G\in\mathcal U_{\mathrm{core}}(G)
\iff
|N_G(H)|\mid 2a_K(H)
\qquad\bigl((K)\in\operatorname{Sub}(G)^c\bigr).
\]
\end{prop}

\begin{proof}
Grouping $G$-conjugate subgroups $L$ in the Gluck--Yoshida formula above, the coefficient of $[G/K]$ is $a_K(H)/|N_G(H)|$. Thus integrality of $2e_H^G$ is equivalent to the stated divisibility conditions. The conclusion follows from Theorem~\ref{thm:core-trivial-idempotents}.
\end{proof}

\begin{ex}[A nontrivial core-trivial unit]\label{ex:s4-core-trivial-unit}
Let $G=S_4$ and let $H\cong S_3$ be a point stabilizer. Then
\[
\operatorname{core}_{S_4}(H)=1,
\qquad
N_{S_4}(H)=H,
\]
so $H$ is non-core-cyclic. Let $C_2\leq H$ be generated by a transposition and let $C_3\leq H$ be generated by a $3$-cycle. In the subgroup lattice of $H$,
\[
\mu_{\operatorname{Sub}(S_4)}(H,H)=1,
\qquad
\mu_{\operatorname{Sub}(S_4)}(C_2,H)
=\mu_{\operatorname{Sub}(S_4)}(C_3,H)=-1,
\qquad
\mu_{\operatorname{Sub}(S_4)}(1,H)=3.
\]
The three subgroups of order $2$ contained in $H$ are conjugate in $S_4$, so the Gluck--Yoshida formula gives
\[
e_H^{S_4}
=
[S_4/H]-[S_4/C_2]-\frac12[S_4/C_3]+\frac12[S_4/1].
\]
Hence $2e_H^{S_4}\in\B{S_4}$, and Theorem~\ref{thm:core-trivial-idempotents} shows that
\[
u
:=
1_{\B{S_4}}-2e_H^{S_4}
=
1_{\B{S_4}}-2[S_4/H]+2[S_4/C_2]+[S_4/C_3]-[S_4/1]
\]
is a nontrivial core-trivial unit.
\end{ex}

\subsection{A linear criterion from Yoshida's criterion}

For a subgroup $K\leq G$, write
\[
W_G(K):=N_G(K)/K.
\]
We shall use Yoshida's unit criterion~\cite[Proposition~6.5]{Yoshida1990}.

For a finite group $W$, define
\[
\operatorname{Cl}_0(W,\F_2)
:=
\left\{
f:W\longrightarrow\F_2
\ \middle|\
f(1)=0,\quad
f(awa^{-1})=f(w)\ \ (a,w\in W)
\right\}.
\]
Under pointwise addition this is an $\F_2$-vector space, and $\operatorname{Hom}(W,\F_2)$ is a subspace.

\begin{dfn}[Local defect space]
\label{dfn:local-yoshida-defect-space}
Define
\[
\mathfrak D(W)
:=
\operatorname{Cl}_0(W,\F_2)
\big/
\operatorname{Hom}(W,\F_2).
\]
For a class function $f$, write $[f]\in\mathfrak D(W)$ for its equivalence class.
\end{dfn}

\begin{prop}[Dimension of the local defect space]
\label{prop:local-yoshida-defect-dimension}
Let $k(W)$ denote the number of conjugacy classes of $W$. Then
\[
\dim_{\F_2}\mathfrak D(W)
=k(W)-1-\dim_{\F_2}\operatorname{Hom}(W,\F_2).
\]
Moreover,
\[
\mathfrak D(W)=0
\iff
W=1\text{ or }W\cong C_2.
\]
\end{prop}

\begin{proof}
By definition,
\[
\dim_{\F_2}\operatorname{Cl}_0(W,\F_2)=k(W)-1,
\]
which gives the dimension formula.

Let
\[
W^{[2]}:=\langle w^2\mid w\in W\rangle.
\]
The quotient $W/[W,W]W^{[2]}$ is an elementary abelian $2$-group and
\[
|W/[W,W]W^{[2]}|
=2^{\dim_{\F_2}\operatorname{Hom}(W,\F_2)}.
\]
The quotient map $W\to W/[W,W]W^{[2]}$ is constant on conjugacy classes, so
\[
k(W)
\geq
2^{\dim_{\F_2}\operatorname{Hom}(W,\F_2)}.
\]
If $\mathfrak D(W)=0$, then
\[
k(W)-1=\dim_{\F_2}\operatorname{Hom}(W,\F_2),
\]
and hence
\[
2^{\dim_{\F_2}\operatorname{Hom}(W,\F_2)}-1
\leq
\dim_{\F_2}\operatorname{Hom}(W,\F_2).
\]
Thus $\dim_{\F_2}\operatorname{Hom}(W,\F_2)$ is $0$ or $1$. In the first case $k(W)=1$, so $W=1$. In the second case $k(W)=2$. If a finite group has exactly two conjugacy classes, then the nonidentity conjugacy class has size $|W|-1$ and must divide $|W|$, whence $|W|=2$ and $W\cong C_2$. Conversely, $\mathfrak D(W)=0$ is immediate for $W=1$ and $W=C_2$.
\end{proof}

For a finite group $G$, put
\[
V_G:=\F_2^{\operatorname{Sub}_{\mathrm{ncc}}(G)^c}.
\]
The characteristic functions $x^S$ defined above are elements of $V_G$, and by our convention $\mathcal S_2(G)$ is viewed as a subspace of $V_G$. For $x\in V_G$, define
\[
\widetilde x_{(H)}
:=
\begin{cases}
x_{(H)},&(H)\in\operatorname{Sub}_{\mathrm{ncc}}(G)^c,\\
0,&(H)\notin\operatorname{Sub}_{\mathrm{ncc}}(G)^c.
\end{cases}
\]

For each $K\leq G$ and $g\in N_G(K)$, put
\[
\lambda_{x,K}(gK)
:=
\widetilde x_{(\langle K,g\rangle)}+\widetilde x_{(K)}.
\]
This is independent of the choice of coset representative, since $\langle K,gk\rangle=\langle K,g\rangle$ for $k\in K$. Moreover, for $a\in N_G(K)$,
\[
\langle K,aga^{-1}\rangle=a\langle K,g\rangle a^{-1},
\]
so $\lambda_{x,K}$ is constant on conjugacy classes of $W_G(K)$. Finally,
\[
\lambda_{x,K}(K)
=\widetilde x_{(K)}+\widetilde x_{(K)}=0.
\]
Therefore
\[
\lambda_{x,K}\in\operatorname{Cl}_0(W_G(K),\F_2),
\]
and the map $x\mapsto\lambda_{x,K}$ is $\F_2$-linear.

Choose one representative $K$ of each conjugacy class $(K)\in\operatorname{Sub}(G)^c$, and put
\[
\mathfrak D_G
:=
\bigoplus_{(K)\in\operatorname{Sub}(G)^c}\mathfrak D(W_G(K)).
\]

\begin{dfn}[Defect map for Yoshida's criterion]
\label{dfn:yoshida-defect-map}
Define
\[
\operatorname{Def}_G:V_G\longrightarrow\mathfrak D_G,
\qquad
x\longmapsto
\bigl([\lambda_{x,K}]\bigr)_{(K)\in\operatorname{Sub}(G)^c}.
\]
\end{dfn}

\begin{thm}[Kernel of the defect map]\label{thm:yoshida-defect-kernel}
\[
\ker\operatorname{Def}_G=\mathcal S_2(G).
\]
Consequently,
\[
0\longrightarrow\mathcal S_2(G)
\longrightarrow V_G
\xrightarrow{\ \operatorname{Def}_G\ }
\operatorname{im}(\operatorname{Def}_G)
\longrightarrow0
\]
is exact, and
\[
\dim_{\F_2}\mathcal S_2(G)
=|\operatorname{Sub}_{\mathrm{ncc}}(G)^c|
-\dim_{\F_2}\operatorname{im}(\operatorname{Def}_G).
\]
\end{thm}

\begin{proof}
The ratio of signs appearing in Yoshida's unit criterion is
\[
\frac{(-1)^{\widetilde x_{(\langle K,g\rangle)}}}{(-1)^{\widetilde x_{(K)}}}
=(-1)^{\lambda_{x,K}(gK)}.
\]
Thus, by Yoshida's criterion, the function
\[
H\longmapsto(-1)^{\widetilde x_{(H)}}
\]
is the mark vector of a Burnside unit if and only if $[\lambda_{x,K}]=0$ for every $K\leq G$. Since $\widetilde x_{(H)}=0$ for every core-cyclic subgroup $H\leq G$, Theorem~\ref{thm:core-trivial-idempotents} shows that this is equivalent to $x\in\mathcal S_2(G)$. The dimension formula follows from the rank-nullity theorem.
\end{proof}

\begin{cor}[A dimension lower bound from local Weyl groups]\label{cor:local-weyl-lower-bound-core-trivial}
\[
\dim_{\F_2}\mathcal S_2(G)
\geq
\max\left\{
0,
\begin{aligned}
&|\operatorname{Sub}_{\mathrm{ncc}}(G)^c|\\
&\quad-
\sum_{(K)\in\operatorname{Sub}(G)^c}
\left(
\begin{aligned}
&k(W_G(K))-1\\
&\quad-\dim_{\F_2}\operatorname{Hom}(W_G(K),\F_2)
\end{aligned}
\right)
\end{aligned}
\right\}.
\]
\end{cor}

\begin{proof}
By Theorem~\ref{thm:yoshida-defect-kernel},
\[
\dim_{\F_2}\operatorname{im}(\operatorname{Def}_G)
\leq\dim_{\F_2}\mathfrak D_G,
\]
where
\[
\dim_{\F_2}\mathfrak D_G
=\sum_{(K)\in\operatorname{Sub}(G)^c}
\left(
\begin{aligned}
&k(W_G(K))-1\\
&\quad-\dim_{\F_2}\operatorname{Hom}(W_G(K),\F_2)
\end{aligned}
\right),
\]
and the result follows from Proposition~\ref{prop:local-yoshida-defect-dimension}.
\end{proof}

\subsection{Splitting of the defect map along normal quotients}

Let $N\trianglelefteq G$, and write the quotient map as
\[
\pi:G\longrightarrow G/N.
\]
For $H\leq G$, write $\pi(H)=HN/N$, and for $L\leq G/N$, put
\[
\widehat L:=\pi^{-1}(L).
\]

\begin{lem}[Core-cyclicity and normal quotients]
\label{lem:core-cyclic-normal-quotient}
The following statements hold.
\begin{enumerate}
 \item If $H\leq G$ is core-cyclic, then $\pi(H)\leq G/N$ is core-cyclic.
 \item For $L\leq G/N$, the subgroup $L$ is core-cyclic in $G/N$ if and only if $\widehat L$ is core-cyclic in $G$.
\end{enumerate}
\end{lem}

\begin{proof}
We have $N\leq\operatorname{core}_G(HN)$ and $\operatorname{core}_G(H)\leq\operatorname{core}_G(HN)$. Hence there is a natural surjection
\[
H/\operatorname{core}_G(H)
\twoheadrightarrow
HN/\operatorname{core}_G(HN).
\]
Moreover,
\[
\operatorname{core}_{G/N}(HN/N)
=\operatorname{core}_G(HN)/N,
\]
so the group on the right is isomorphic to $\pi(H)/\operatorname{core}_{G/N}(\pi(H))$. This proves the first assertion.

For the second assertion, since $\widehat L\geq N$,
\[
\operatorname{core}_{G/N}(L)
=\operatorname{core}_G(\widehat L)/N,
\]
and therefore
\[
\widehat L/\operatorname{core}_G(\widehat L)
\cong
L/\operatorname{core}_{G/N}(L).
\]
\end{proof}

By Lemma~\ref{lem:core-cyclic-normal-quotient}, define a linear map
\[
\operatorname{Inf}:V_{G/N}\longrightarrow V_G
\]
by
\[
\bigl(\operatorname{Inf}(x)\bigr)_{(H)}
:=
\begin{cases}
x_{(\pi(H))},
&\pi(H)\in\operatorname{Sub}_{\mathrm{ncc}}(G/N),\\
0,
&\pi(H)\notin\operatorname{Sub}_{\mathrm{ncc}}(G/N),
\end{cases}
\]
for each $(H)\in\operatorname{Sub}_{\mathrm{ncc}}(G)^c$. Define
\[
\operatorname{Fix}:V_G\longrightarrow V_{G/N}
\]
by
\[
\bigl(\operatorname{Fix}(y)\bigr)_{(L)}
:=y_{(\widehat L)}
\]
for each $(L)\in\operatorname{Sub}_{\mathrm{ncc}}(G/N)^c$. By construction,
\[
\operatorname{Fix}\circ\operatorname{Inf}
=\operatorname{id}_{V_{G/N}}.
\]

For $H\leq G$, the quotient map induces a group homomorphism
\[
\overline\pi_H:W_G(H)\longrightarrow W_{G/N}(\pi(H)),
\qquad
gH\longmapsto\pi(g)\pi(H).
\]
Pullback induces
\[
\overline\pi_H^*:
\mathfrak D(W_{G/N}(\pi(H)))
\longrightarrow
\mathfrak D(W_G(H)).
\]

Choose representatives of subgroup conjugacy classes so that, for each representative $L$ in $G/N$, its inverse image $\widehat L$ is among the chosen representatives in $G$. Using the identifications induced by conjugation isomorphisms of Weyl groups, define linear maps on defect spaces, denoted by the same symbols,
\[
\operatorname{Inf}:\mathfrak D_{G/N}\longrightarrow\mathfrak D_G,
\qquad
\bigl(\operatorname{Inf}(\omega)\bigr)_H
:=\overline\pi_H^*(\omega_{\pi(H)}).
\]

For $L\leq G/N$,
\[
N_G(\widehat L)=\pi^{-1}(N_{G/N}(L)),
\]
so
\[
\overline\pi_{\widehat L}:
W_G(\widehat L)\xrightarrow{\sim}W_{G/N}(L)
\]
is an isomorphism. Define
\[
\operatorname{Fix}:\mathfrak D_G\longrightarrow\mathfrak D_{G/N},
\qquad
\bigl(\operatorname{Fix}(\eta)\bigr)_L
:=(\overline\pi_{\widehat L}^*)^{-1}(\eta_{\widehat L}).
\]
Then
\[
\operatorname{Fix}\circ\operatorname{Inf}
=\operatorname{id}_{\mathfrak D_{G/N}}.
\]

\begin{thm}[Split naturality along normal quotients]
\label{thm:yoshida-defect-normal-quotient-splitting}
The following diagram commutes:
\[
\begin{tikzcd}[column sep=large,row sep=large]
V_{G/N}
\arrow[r,"\operatorname{Def}_{G/N}"]
\arrow[d,shift left=.7ex,"\operatorname{Inf}"]
&
\mathfrak D_{G/N}
\arrow[d,shift left=.7ex,"\operatorname{Inf}"]
\\
V_G
\arrow[r,"\operatorname{Def}_G"]
\arrow[u,shift left=.7ex,"\operatorname{Fix}"]
&
\mathfrak D_G
\arrow[u,shift left=.7ex,"\operatorname{Fix}"]
\end{tikzcd}
\]
\end{thm}

\begin{proof}
Let $x\in V_{G/N}$, $H\leq G$, and $g\in N_G(H)$. Then
\[
\begin{aligned}
\lambda^G_{\operatorname{Inf}(x),H}(gH)
&=\widetilde x_{(\pi(\langle H,g\rangle))}
  +\widetilde x_{(\pi(H))}\\
&=\widetilde x_{(\langle\pi(H),\pi(g)\rangle)}
  +\widetilde x_{(\pi(H))}\\
&=\lambda^{G/N}_{x,\pi(H)}\bigl(\overline\pi_H(gH)\bigr).
\end{aligned}
\]
Hence, in each local component,
\[
[\lambda^G_{\operatorname{Inf}(x),H}]
=\overline\pi_H^*([\lambda^{G/N}_{x,\pi(H)}]),
\]
which proves commutativity in the downward direction.

Now let $y\in V_G$, $L\leq G/N$, and $g\in N_G(\widehat L)$. Since $N\leq\widehat L$,
\[
\pi^{-1}(\langle L,\pi(g)\rangle)
=\langle\widehat L,g\rangle.
\]
Therefore
\[
\lambda^{G/N}_{\operatorname{Fix}(y),L}(\pi(g)L)
=
\lambda^G_{y,\widehat L}(g\widehat L).
\]
Since $\overline\pi_{\widehat L}$ is an isomorphism, passing to defect classes gives commutativity in the upward direction.
\end{proof}

\begin{thm}[Decomposition along a normal quotient]
\label{thm:yoshida-defect-normal-quotient-decomposition}
With respect to the direct-sum decompositions
\[
V_G
=\operatorname{Inf}(V_{G/N})\oplus\ker\operatorname{Fix},
\]
\[
\mathfrak D_G
=\operatorname{Inf}(\mathfrak D_{G/N})
\oplus\ker\operatorname{Fix},
\]
one has
\[
\operatorname{Def}_G
\cong
\operatorname{Def}_{G/N}
\oplus
\left.\operatorname{Def}_G\right|_{\ker\operatorname{Fix}}.
\]
Consequently,
\[
\mathcal S_2(G)
=
\operatorname{Inf}\bigl(\mathcal S_2(G/N)\bigr)
\oplus
\bigl(\mathcal S_2(G)\cap\ker\operatorname{Fix}\bigr),
\]
and, in particular,
\[
\dim_{\F_2}\mathcal S_2(G)
=
\dim_{\F_2}\mathcal S_2(G/N)
+
\dim_{\F_2}\bigl(\mathcal S_2(G)\cap\ker\operatorname{Fix}\bigr).
\]
\end{thm}

\begin{proof}
Each of the two maps denoted by $\operatorname{Fix}$ is a left inverse to the corresponding $\operatorname{Inf}$, giving the displayed direct-sum decompositions. By Theorem~\ref{thm:yoshida-defect-normal-quotient-splitting}, the map $\operatorname{Def}_G$ sends the inflation component to the inflation component and $\ker\operatorname{Fix}$ to $\ker\operatorname{Fix}$. Thus the stated decomposition of the defect map follows. The remaining assertions follow by taking kernels componentwise.
\end{proof}

\begin{cor}[Splitting of core-trivial units along a normal quotient]
\label{cor:ucore-quotient-split}
Let $N\trianglelefteq G$. Then
\[
\mathcal U_{\mathrm{core}}(G)
\cong
\operatorname{Inf}_{G/N}^G
\bigl(\mathcal U_{\mathrm{core}}(G/N)\bigr)
\times
\bigl(
\mathcal U_{\mathrm{core}}(G)
\cap\ker((- )^N)
\bigr).
\]
\end{cor}

\begin{proof}
Transport the decomposition of $\mathcal S_2$ in Theorem~\ref{thm:yoshida-defect-normal-quotient-decomposition} through the correspondence
\[
S\longmapsto1_{\B G}-2e_S
\]
of Theorem~\ref{thm:core-trivial-idempotents}.
\end{proof}

\subsection{Decomposition of the defect map along direct products}

Let $G_1,\ldots,G_m$ be finite groups, and put
\[
G:=\prod_{i=1}^{m}G_i,
\qquad
\pi_i:G\longrightarrow G_i,
\qquad
N_i:=\ker\pi_i=\prod_{j\neq i}G_j.
\]
Via the quotient isomorphism $G/N_i\cong G_i$, denote the maps defined for normal quotients by
\[
\operatorname{Inf}_i:V_{G_i}\longrightarrow V_G,
\qquad
\operatorname{Fix}_i:V_G\longrightarrow V_{G_i},
\]
and
\[
\operatorname{Inf}_i:\mathfrak D_{G_i}\longrightarrow\mathfrak D_G,
\qquad
\operatorname{Fix}_i:\mathfrak D_G\longrightarrow\mathfrak D_{G_i}.
\]

\begin{lem}[Vanishing of cross terms between distinct factors]
\label{lem:direct-product-cross-vanishing}
On both $V_G$ and the defect spaces, if $i\neq j$, then
\[
\operatorname{Fix}_i\circ\operatorname{Inf}_j=0.
\]
\end{lem}

\begin{proof}
Let $x\in V_{G_j}$ and let $L\leq G_i$ be non-core-cyclic. The inverse image of $L$ under $\pi_i$ is
\[
\widehat L_i=L\times\prod_{k\neq i}G_k.
\]
If $i\neq j$, then $\pi_j(\widehat L_i)=G_j$. Hence
\[
\bigl(\operatorname{Fix}_i\operatorname{Inf}_j(x)\bigr)_{(L)}
=\widetilde x_{(G_j)}=0,
\]
because $G_j$ is core-cyclic.

For the defect spaces, applying $\operatorname{Inf}_j$ to the $\widehat L_i$-component read by $\operatorname{Fix}_i$ gives a pullback from the $G_j$-component. But
\[
W_{G_j}(G_j)=1,
\qquad
\mathfrak D(1)=0,
\]
so this component is zero.
\end{proof}

\begin{thm}[Direct-sum decomposition along direct products]
\label{thm:yoshida-defect-direct-product-decomposition}
With respect to the direct-sum decompositions
\[
V_G
=\left(\bigoplus_{i=1}^{m}
  \operatorname{Inf}_i(V_{G_i})\right)
 \oplus\bigcap_{i=1}^{m}\ker\operatorname{Fix}_i,
\]
\[
\mathfrak D_G
=\left(\bigoplus_{i=1}^{m}
  \operatorname{Inf}_i(\mathfrak D_{G_i})\right)
 \oplus\bigcap_{i=1}^{m}\ker\operatorname{Fix}_i,
\]
one has
\[
\operatorname{Def}_G
\cong
\left(\bigoplus_{i=1}^{m}\operatorname{Def}_{G_i}\right)
\oplus
\left.\operatorname{Def}_G\right|_{\bigcap_{i=1}^{m}\ker\operatorname{Fix}_i}.
\]
Consequently,
\[
\mathcal S_2(G)
=\left(\bigoplus_{i=1}^{m}
  \operatorname{Inf}_i\bigl(\mathcal S_2(G_i)\bigr)\right)
\oplus
\left(\mathcal S_2(G)\cap\bigcap_{i=1}^{m}\ker\operatorname{Fix}_i\right),
\]
and, in particular,
\[
\dim_{\F_2}\mathcal S_2(G)
=\sum_{i=1}^{m}\dim_{\F_2}\mathcal S_2(G_i)
 +\dim_{\F_2}\left(\mathcal S_2(G)\cap\bigcap_{i=1}^{m}\ker\operatorname{Fix}_i\right).
\]
\end{thm}

\begin{proof}
For each $i$,
\[
\operatorname{Inf}_i\operatorname{Fix}_i:V_G\longrightarrow V_G
\]
is idempotent, and Lemma~\ref{lem:direct-product-cross-vanishing} gives, for $i\neq j$,
\[
\operatorname{Inf}_i\operatorname{Fix}_i
\operatorname{Inf}_j\operatorname{Fix}_j=0.
\]
Thus for every $x\in V_G$,
\[
x+\sum_{i=1}^{m}\operatorname{Inf}_i\operatorname{Fix}_i(x)
\in
\bigcap_{i=1}^{m}\ker\operatorname{Fix}_i,
\]
which gives the first direct-sum decomposition. Directness follows by applying the maps $\operatorname{Fix}_i$. The same argument applies to $\mathfrak D_G$.

By naturality along normal quotients, $\operatorname{Def}_G$ sends each inflation component to the corresponding inflation component, and its restriction there agrees with $\operatorname{Def}_{G_i}$. It also sends the residual component
\[
\bigcap_{i=1}^{m}\ker\operatorname{Fix}_i
\]
in $V_G$ into the corresponding residual component of $\mathfrak D_G$. Hence $\operatorname{Def}_G$ decomposes as stated. The formula for its kernel follows componentwise.
\end{proof}

\begin{cor}[Factor and residual components of core-trivial units for direct products]
\label{cor:core-unit-direct-product-defect-decomposition}
\[
\mathcal U_{\mathrm{core}}(G)
\cong
\left(\prod_{i=1}^{m}\mathcal U_{\mathrm{core}}(G_i)\right)
\times
\bigcap_{i=1}^{m}
\ker\left(
(-)^{N_i}:\mathcal U_{\mathrm{core}}(G)
\longrightarrow\mathcal U_{\mathrm{core}}(G_i)
\right).
\]
\end{cor}

\begin{proof}
Under the correspondence $S\mapsto1_{\B G}-2e_S$ from Theorem~\ref{thm:core-trivial-idempotents} and the product formula of Lemma~\ref{lem:idempotent-symmetric-difference}, inflation and fixed points correspond to $\operatorname{Inf}_i$ and $\operatorname{Fix}_i$, respectively. The assertion follows from the kernel decomposition in Theorem~\ref{thm:yoshida-defect-direct-product-decomposition}.
\end{proof}

By Example~\ref{ex:s4-core-trivial-unit},
\[
\mathcal U_{\mathrm{core}}(S_4)\neq\{1\}.
\]
Therefore, for every $m\geq1$,
\[
\operatorname{rank}\bigl(\mathcal U_{\mathrm{core}}(S_4^m)\bigr)\geq m.
\]
Indeed, the $m$ units obtained by inflating the nontrivial core-trivial unit of Example~\ref{ex:s4-core-trivial-unit} from the individual factors are independent by Corollary~\ref{cor:core-unit-direct-product-defect-decomposition}. Hence the rank of $\mathcal U_{\mathrm{core}}(G)$ is unbounded among finite groups.

\end{document}